\documentclass[11pt, A4paper]{amsart}
\pdfoutput=1
\usepackage[foot]{amsaddr}
\usepackage{enumerate}
\usepackage{latexsym}
\usepackage{ascmac}
\usepackage{amssymb}
\usepackage{amsmath}
\usepackage{oldgerm}
\usepackage{amsthm}
\usepackage{mathtools}
\mathtoolsset{showonlyrefs=true}
\usepackage{setspace}
\usepackage{ulem} 
\usepackage{tikz}
\usepackage{tikz-qtree}
\usetikzlibrary{patterns,positioning,fit,calc,circuits.ee.IEC,circuits.logic.US,circuits.logic.IEC,math}
\theoremstyle{plain}\newtheorem{thm}{Theorem}
\theoremstyle{plain}\newtheorem{known}{Theorem}

\theoremstyle{plain}\newtheorem{cor}{Corollary}
\theoremstyle{plain}\newtheorem{prp}{Proposition}
\theoremstyle{plain}\newtheorem{lem}{Lemma}
\theoremstyle{remark}\newtheorem{rem}{Remark}
\theoremstyle{plain}\newtheorem{proc}{Procedure}
\begin{document}
\newcommand{\mychi}{\mbox{1}\hspace{-0.35em}\mbox{1}}
\def \Mod {\operatorname{Mod}}
\newcommand \Th[1]{{#1}th{\,}}
\title{A congruential recurrence characterizes the inverses of S\'{o}s permutations}
\author{Makoto Nagata$^{1)}$}
\address{1) Faculty of Pharmacy, Osaka Medical and Pharmaceutical University}
\author{Yoshinori Takei$^{2)}$}
\address{2) Faculty of Sport Science, Nippon Sport Science University}
\maketitle
%
%
\newcommand{\mysos}{S\'{o}s}
\newcommand{\mysur}{Sur\'{a}nyi}
\newcommand{\changedred}[1]{\textcolor{red}{#1}}
\newcommand{\supermod}[1]{\overline{\mathrm{Mod}}_{#1}}
\newcommand{\ryakuki}[1]{({#1})_m}
\newcommand{\supermodtwo}[2]{\overline{\mathrm{Mod}}_{#1}({#2})}
\def \Sym {\frak{S}}
\def \Sos {\mathcal{S}}
\def \Sinv {\mathcal{S}^{*}}
\def \Rcr {\mathcal{V}}
\def \Req {\mathcal{W}}
\def \pNAP {\mathcal{X}}
\def \Deq {\mathcal{Y}}
\def \myPhi {\Psi}
\begin{abstract}
  In a proof of the three gaps theorem, a class of permutations
  known as the {\mysos} permutations was introduced. It is
  known that a {\mysos} permutation, as a sequence, satisfies
  a certain recurrence ({\mysos}'s recurrence),
  however, whether the converse holds remains unknown.
  On the other hand, the inverses of {\mysos} permutations
  have been studied also. It has been reported that
  such a permutation satisfies a congruential recurrence as
  a sequence. The converse problem of this fact, i.e.,
  whether a permutation satisfying the congruential recurrence
  is the inverse of a {\mysos} permutation, is also unsolved,
  except for a finite number of the degrees of the permutations.
  This paper relates the set of permutations satisfying
  the congruential recurrence to other sets of permutations
  and gives upper bounds for their cardinalities. The upper bounds
  are in fact tight. In particular, the set of 
  permutations satisfying the congruential recurrence
  has the same cardinality as that of {\mysos} permutations,
  giving the affirmative answer to the above unsolved problem as
  a corollary.
As another corollary, it is shown that all permutation, which is 
regarded as a {\it congruential quasi-progression of diameter $1$\/} in 
that the set formed by the first order differences modulo the degree 
of the permutation is a singleton or a set of two successive 
integers, is the inverse of a {\mysos} permutation with an elementary 
operation called as shift applied.
As an application of these
  facts, we present a procedure that lifts the set of the inverses of the
  {\mysos} permutations of a given degree to the
  set of the same kind with the degree increased by one,
  without referring to the underlying parameters defining 
  the {\mysos} permutations or the Farey sequence
  associated to them.
\end{abstract}

\section{Introduction}
  For a positive integer $m$ and a real number $\alpha,$
  the permutation that sorts the fractional parts
  of $m$ real numbers $\alpha, 2\alpha, \ldots, m\alpha$
  in the increasing order is referred to as a {\it {\mysos} permutation\/}
  of degree $m$. This class of permutations
  was introduced in \cite{sos} to prove
  so-called three gaps theorem and \cite{bock} describes {\mysos} permutations
  in detail. A known property of the {\mysos} permutations
  is that their terms satisfy a certain recurrence
  \cite[Theorem I]{sos}, to which we refer as {\mysos}'s recurrence
  in what follows. Another known fact is the existence
  of {\it {\mysur}'s bijection\/} \cite[Satz I]{sur},
  which maps
  the pair formed by the denominators of two successive
  terms in the \Th{$m$} Farey sequence to the pair
  of the first and last terms
  of a {\mysos} permutation of degree $m$.
  Thus, if a permutation satisfies {\mysos}'s recurrence as a sequence
  {\it and if\/} its first and last terms form the pair of
  the denominators of two successive terms in the \Th{$m$} Farey sequence, then it is a {\mysos} permutation.
However, it is not known
whether a permutation satisfying {\mysos}'s
  recurrence
  is always a {\mysos} permutation.

  The {\it inverses\/} of the {\mysos} permutations have
  been also studied \cite{tikan1,tikan2, ranking2}
  and it is reported that they satisfy a certain congruential
  recurrence. The converse problem of this fact, i.e.,
  whether a permutation satisfying the congruential
  recurrence is 
 the inverse of a {\mysos} permutation, is still unsolved, except
  that it has been verified using computers
  for a finite number of the degrees $m$.

  This paper presents an upper bound of the cardinality
  of the set of the permutations satisfying the congruential recurrence,
  by relating the set to another set of permutations which has
  a different defining relation and an increased degree.
  In fact the upper bound is the same as the number of {\mysos}
  permutations which is already known, as a corollary,
  we obtain the affirmative
  answer to the converse problem that all permutations
  satisfying the congruential recurrence are 
  the inverses of {\mysos} ones.

We also obtain another corollary which states, roughly speaking, 
that all permutation which forms a {\it congruential quasi-progression of diameter $1$\/}
essentially comes from the inverse of a {\mysos} permutation.

  Furthermore, as an application, we present a {\it lifting procedure\/}
  which lifts the set of permutations satisfying the congruential
  recurrence of degree $m-1$ to the set of the same kind whose degree
  is $m$. 
Though the resulting set of our procedure is
  the set of 
 the inverses of 
{\mysos} permutations of degree $m$, 
it is not a product made from the underlying parameters
$\alpha$ of {\mysos} permutations
or the Farey sequence. 
Only the information of each sequence in the set of degree $m-1$ and integer operations are used by the procedure.

  The rest of the paper is organized as follows.
  Section 2 is preliminary, where we introduce notations and known results
  on which we will depend. In Section 3, the main results and
  their proofs are presented. In Section 4, the aforementioned
  application is described.

\section{Some known results on the inverses of {\mysos} permutations}\label{sect:known}
In this section, we recall some known properties about the inverses of  {\mysos} permutations.

Suppose that $m$ is an integer $\ge 2$.
  Let $[m]$ denote the set $\{1,2,\ldots, m\}$ of consecutive integers
  from $1$ to $m$. In this paper,  the term ``permutation''
  means a bijective map from $[m]$ to $[m]$ and $m$
  is referred to as the {\it degree\/} of the permutation.
  Let $\Sym_m$ denote the set of all permutations of degree $m$.
  For a real number $\alpha$, its integral part $\lfloor \alpha \rfloor$
  denotes the largest integer that is not larger than $\alpha$, whereas
  $\{ \alpha \}$ 
  denotes
  the fractional part $\alpha - \lfloor \alpha \rfloor$ of $\alpha.$

  In this paper 
  referring to the name of the author of \cite{sos},
  a permutation $\sigma$ of degree $m$ is said to be a
  {\it {\mysos} permutation\/}, if there exists a real number $\alpha$
  such that
\begin{equation}\label{Sosnojyouken1}
0< \{\sigma(1)\alpha\} < \{\sigma(2)\alpha\} < \cdots < \{\sigma(m)\alpha\}<1.
\end{equation}
  Hereafter the set of all {\mysos} permutations
  of degree $m$ is denoted by ${\Sos}_m$. In addition,
  let ${\Sinv}_m$ be the collection of the inverses of them:
  $${\Sinv}_m:=\{\sigma^{-1}\in {\Sym}_m\ : \ \sigma \in {\Sos}_m\}.$$
  For the properties of {\mysos} permutations, the readers
  are referred to \cite{bock} as well as \cite{sos,sur,rosia}.
  In particular,
  {\mysur}'s bijection \cite[Satz I]{sur} describes
    the crucial correspondence between a {\mysos} permutation
    $\sigma\in{\Sym}_m$ satisfying Ineq.~\eqref{Sosnojyouken1}
    and the interval formed by consecutive two terms in the
    Farey sequence into which the real parameter $\alpha$ falls.
  Since the main interest of the current paper is in {\it the inverses\/}
  of the {\mysos} permutations rather than themselves, below,
  the description
  of the facts/properties of the {\mysos} permutations is kept to the
  minimum necessary.

  For a Boolean predicate $P$, let
  $$
  \mychi(P) := \begin{cases} 1 & \text{\ if } P \text{\ is true,}
    \\ 0 & \text{\ if } P \text{\ is false} \end{cases}
  $$
  be its truth value converted to the integer (i.e.,
  another notation of Iverson's bracket).
  Now we quote a few known facts on the {\mysos} permutations and their
  inverses.
\begin{known}[{\cite[Theorem I]{sos}}]
\label{Known0}
Suppose that $m\ge 2$ and suppose that
$\sigma$ is 
a {\mysos} permutation of degree $m$.
Then 
{
\begin{equation}\label{sosrec}
\sigma(i+1)= \sigma(i) +\begin{cases}
\sigma(1) & \text{\ if\ }\  \sigma(i) \le m-\sigma(1) ,\\
\sigma(1)-\sigma(m) & \text{\ if\ }\ m-\sigma(1)< \sigma(i) < \sigma(m) ,\\
-\sigma(m) & \text{\ if\ }\ \sigma(m)\le \sigma(i)
\end{cases}
\end{equation}
}
holds for $i\in[m-1]$.
\end{known}

It remains unknown whether the converse, 
i.e., 
a permutation
  satisfying the recurrence \eqref{sosrec} is a {\mysos} permutation, is true.

  On the other hand, the inverses of {\mysos} permutations
  satisfy a congruential recurrence.
\begin{known}[{cf.~{\cite[Corollary 1]{tikan1}}}]
\label{Known1}
Suppose that $m\ge 2$ and suppose that 
$\theta$ is the inverse of a {\mysos} permutation of degree $m$.
Then
{\rm
\begin{equation}\label{knowneq1}
\theta(i+1)-\theta(i)\equiv \theta(1)-\mychi(\theta(m)\le \theta(i)) \mod m
\end{equation}
}
holds for $i\in[m-1]$.
\end{known}
  In other words, the inverse of a {\mysos}
  permutation is a slightly generalized congruential arithmetic progression
  in that the differences $\theta(i+1)-\theta(i)$ take, in modulo $m$,
  at most two values $\theta(1)$ and $\theta(1)-1$.
  The first question in this paper is whether the converse
  of Theorem \ref{Known1} is the case.
  Eq.~\eqref{knowneq1}, a requirement on the
  congruential difference set of a permutation $\theta$,
  is an interesting property of a permutation in its own right
  and we define the following set of permutations, not necessarily {\mysos} ones,
  with this property.
For $m \ge 2$, a subset ${\Rcr}_m \subset {\Sym}_m$ is defined as
$${\Rcr}_m:=\{\theta\in {\Sym}_m \ : \ \theta(i+1)-\theta(i)\equiv \theta(1)-\mychi(\theta(m)\le \theta(i)) \mod m  \ \text{for} \ i\in[m-1]\}.$$
Then, Theorem \ref{Known1} is shortly described as
$${\Sinv}_m\subset {\Rcr}_m
.$$

  We introduce an action $\lambda$
  over permutations and an equivalence relation
  based on $\lambda$.
  Let $\Mod_m(j)$ be the standard residue of an integer $j$ modulo $m$
  taking the value in the set $\{0, 1, \ldots, m-1\}$. 
  In addition, let
  $\supermod{m}(j):=\Mod_m(j-1)+1$ be a variant of it,
  in which $0$ of the range is replaced with $m$
  keeping the {$\operatorname{mod}$ $m$} class.
  For $\theta \in {\Sym}_m$, the map
  \begin{equation}\label{sihuto1teigisiki}
   \lambda(\theta) := \supermod{m}(\theta(\cdot) + 1)
  \end{equation}
   is also an element of ${\Sym}_m$ (We note that
   in \cite[Lemma 1]{bock} the same concept as $\lambda$ is introduced
   with the notation $c^{-1}$).
   The action
   $\lambda: {\Sym}_m \ni \theta \mapsto \lambda(\theta) \in {\Sym}_m$
   forms a cyclic group $\langle \lambda \rangle$ by iterative applications
   $\lambda \circ \cdots \circ \lambda(\theta)$. It is easy to check
   that $\lambda^k(\theta) = \supermod{m}(\theta(\cdot) + k)$ for
   $k \in \mathbb{Z}$
   and that
   $\lambda^{m}(\theta) = \theta$ for all $\theta \in {\Sym}_m$,
   while $\lambda^{k}(\theta) \neq \lambda^{l}(\theta)$ for
   $0 \leq k < l < m$ considering that $\lambda^{k}(\theta)(1) \ (k=0, \ldots, m-1)$
   take the $m$ different values of $[m]$.
   In other words, for any $\theta \in {\Sym}_m$
   the orbit $\langle \lambda \rangle \cdot \theta$ always contains
   $m$ different permutations.
   Any permutation $\lambda^{k}(\theta)$ in the orbit is referred to
   as a {\it shift\/} of $\theta$. Then, for
   $\theta, \theta' \in {\Sym}_m$, a binary relation
   $$
   \theta \sim \theta' \ \overset{\text{def}}{\Leftrightarrow}\ 
   \theta' \text{\ is a shift of \ } \theta
   $$
   is defined. Observe that it has equivalent definitions
   \begin{equation}\label{sihutodoutiteigino2}
   \theta \sim \theta' \ \Leftrightarrow\ \exists\ k \in [m]  \text{\ s.t. }
   \theta'(i) \equiv \theta(i) + k \mod m
\text{\ for\ } i \in [m]
   \end{equation}
   or
   $$
   \theta \sim \theta' \ \Leftrightarrow \ 
   \theta(i) - \theta(j) \equiv
   \theta'(i) - \theta'(j) \mod m \text{\ for\ } i, j \in [m]
   $$
   and that $\sim$ is an equivalence relation over ${\Sym}_m$.
   We say that $\theta$ and $\theta'$ are {\it shift-equivalent\/}
   if $\theta \sim \theta'$.
   Given an arbitrary subset $T$ of permutations, its closure with respect
   to the equivalence relation $\sim$ is defined. Formally,
   the {\it shift-closure operator\/} 
   $\widetilde{\cdot}: 2^{{\Sym}_m}\to 2^{{\Sym}_m}$ is defined
   through
   $$
   \widetilde{T}:=
\{\theta'\in {\Sym}_m\ : \ \exists \ \theta\in T \text{\ s.t. } \theta\sim\theta'\}, \text{\ for \ } T \subset {\Sym}_m.
$$   
Of course, it holds that $\widetilde{\widetilde{T}}=\widetilde{T}.$
We also note that $\{\theta  \in \widetilde{T} : \theta(1) = 1 \}$
is a complete representative system of $\widetilde{T}/\sim,$
because $\widetilde{\{\theta\}} = \langle \lambda \rangle \cdot \theta$
contains $m$ different permutations whose values $\{ \lambda^k(\theta)(1) :
 k = 0, 1, \ldots m-1\}$ coincide with
the set $[m]$ for any $\theta \in {\Sym}_m$.

\begin{rem}
It seems appropriate to mention here that
  the term {\it {\mysos} permutation \/} in \cite{bock} refers
  to a
  bijection $\pi: \{0, \ldots, m-1 \} \rightarrow \{0, \ldots, m-1 \}$
  satisfying
  $0< \{\pi(0)\alpha+\beta\} < \{\pi(1)\alpha+\beta\} < \cdots < \{\pi(m-1)\alpha+\beta\}<1$ for some $\alpha$ and $\beta$,
  whereas in \cite{tikan1}, a permutation
  $\sigma \in {\Sym}_m$ satisfying
  $0< \{\sigma(1)\alpha+\beta\} < \{\sigma(2)\alpha+\beta\} < \cdots < \{\sigma(m)\alpha+\beta\}<1$ for some $\alpha$ and $\beta$
  is referred to as a {\it permutation of {\mysos}-type}.
  They define a class of permutations
  wider than that of \cite{sos}
  using Ineq.~\eqref{Sosnojyouken1},
  by the presence of $\beta$. In \cite{tikan1},
  the distinction of the original and the wider concepts
  is made by using the term {\mysos} permutations and {\mysos}-type
  permutations respectively. This paper also inherits this distinction
  of the terminology. Therefore, a {\mysos} permutation
  means a permutation of {\mysos}-type with $\beta = 0$. 
  By adding the $\beta$ term, a {\mysos} permutation is {\it rotated\/}
  to form a permutation of {\mysos}-type. Taking the inverse,
  it turns out that the inverses of the permutations of {\mysos}-type
  coincide with $\widetilde{{\Sinv}_m}$.
In the following section, it will be revealed that
  our approach
  essentially requires introducing both of
  ${\Sinv}_m$ and $\widetilde{{\Sinv}_m}$.
\end{rem}

\begin{rem}
  The above
  definition of {\mysos}-type permutations
  may remind some readers
  of {\it Beatty sequences\/} $(\lfloor i\alpha + \beta\rfloor)_{i=1}^\infty$
  and/or {\it Sturmian words.\/} Indeed,
  \cite{obr} relates Sturmian words and {\mysos} permutations.
  In contrast, we restrict ourselves to the permutations of a finite
  degree, in this paper.
\end{rem}

 In the following, $\phi$ denotes Euler's totient function,
  i.e., $\phi(x) = |\{ a \in [x]: \gcd(a, x) = 1 \}|$ for any positive
  integer $x$. Aforementioned {\mysur}'s bijection
  \cite{sur} (see also \cite{rosia}) relates
  the pair formed by the denominators of two successive
  terms in the \Th{$m$} Farey sequence to the pair
  $(\sigma(1), \sigma(m))$ of the first and last terms
  of a {\mysos} permutation $\sigma$. An important and direct
  consequence of this one-to-one
  correspondence is the cardinality formula
  $|{\Sos}_m|=\sum_{k=1}^{m}\phi(k)$. Taking the inverse, we have the
  following.
\begin{known}[\cite{sur,rosia}]\label{Known2}
Suppose that  $m\ge 2$. Then the cardinality of ${\Sinv}_m$ is 
$|{\Sinv}_m|=\sum_{k=1}^{m}\phi(k)$.
\end{known}

Furthermore, the cardinality of $\widetilde{{\Sinv}_m}$ 
has been obtained as follows:
\begin{known}[{\cite[Theorem 4]{bock} see also \cite[Theorem 5]{tikan1}}]\label{Known3}
Suppose that  $m\ge 2$. Then the cardinality of $\widetilde{{\Sinv}_m}$ is 
$|\widetilde{{\Sinv}_m}|=m\sum_{k=1}^{m-1}\phi(k)$.
\end{known}
  We note that the assertion \cite[Theorem 4]{bock} was stated
  for {\mysos}-type permutations, i.e.,
  the inverses of the permutations in $\widetilde{{\Sinv}_m}$.
  Also, note that $|\widetilde{{\Sinv}_m}|$ is not
  $m$ times $|{\Sinv}_m|$.

  Lastly, we quote a result with respect to $\widetilde{{\Rcr}_m}$
  from \cite{tikan2}: Let
  $${\pNAP}_m:=\{\theta\in {\Sym}_m\ : \ \exists\ k \in [m-1]\ \forall\ i\in [m-1]\ \big[ \Mod_m(\theta(i+1)-\theta(i))\in\{ k,k+1\}\big]\},$$
  for which ${\pNAP}_m=\widetilde{{\pNAP}_m}$ obviously holds.
An element of ${\pNAP}_m$ may be regarded as a modulo-$m$ version of
an {\it $m$-term quasi-progression of diameter $1$. \/} 
An $m$-term sequence $x_1 < \cdots < x_m$ is said to be an
{\it $m$-term quasi-progression of diameter $d$\/} \cite{BEF}
if there exists
$N$ such that $N \le x_{i+1}-x_i \le N+d$ for $i \in [m-1]$.
We borrow this term 
removing the increasing condition of the sequence, restricting the range of the sequence
to $[m]$ and applying $\Mod_m$ to
the differences.
 We say that $x_1, \ldots, x_m \in [m]$ is
a {\it mod-$m$ congruential $m$-term quasi-progression of diameter $d$\/} if 
there exists
an integer $N$ such that $N \le \Mod_m(x_{i+1}-x_i) \le N+d$ for $i \in [m-1]$.
Using this term, we can say that
${\pNAP}_m$ consists of all permutations of degree $m$ which are mod-$m$ congruential $m$-term quasi-progressions of diameter $1$. 
As we noted
  just after introducing Theorem \ref{Known1}, $\theta \in {\Rcr}_m$ satisfies the stricter condition that the differences $\Mod_m(\theta(i+1)-\theta(i))$ are $\theta(1)-\mychi(\theta(m)\le \theta(i))$. Therefore it holds that ${\Rcr}_m\subset {\pNAP}_m$, from which the inclusion $\widetilde{{\Rcr}_m}\subset{\pNAP}_m$ follows by ${\pNAP}_m=\widetilde{{\pNAP}_m}$.
  In fact, \cite[Corollary 5]{tikan2} asserted the reverse inclusion:
\begin{known}[{\cite[Corollary 5]{tikan2}}]\label{Known5}
Suppose that $m\ge 2$. Then the set ${\pNAP}_m$ coincides with $\widetilde{{\Rcr}_m}$.
\end{known}

  In the next section, our main results will be shown
  without depending on the known results quoted above,
  then the corollaries will be shown using both the main results
  and the known results quoted above.

\section{Main results}\label{sect:mainresults}
Before stating our results, we need to introduce two sets, each of which consists
of
permutations that satisfy their own difference equations.
For $m\ge 2$, 
let
$${\Req}_m:=\{\theta\in {\Sym}_m \ : \ 
\theta(i+1)-\theta(i) =  \theta(1)-\mychi(\theta(m)\le \theta(i))
+m \left(\mychi(\theta(i)\le \theta(i+1)) -1\right)
\ \text{for} \ i\in[m-1]\},
$$
which
satisfies the inclusion ${\Req}_m\subset {\Rcr}_m$ clearly.
Next, for $m\ge 3$, 
let
$${\Deq}_m:=\{\theta\in {\Sym}_m \ : \ \Delta_\theta(i+1)-\Delta_\theta(i)=0 \ \text{for} \ i\in[m-2]\},$$
introducing the notation
\begin{equation}\label{Deltateigisiki}
\Delta_\theta(i):=
\theta(i+1)-\theta(i)
+\mychi(\theta(m)\le \theta(i))
-\mychi(\theta(1)\le \theta(i+1))
-(m-1)\mychi(\theta(i)\le \theta(i+1))
\ \text{for\ }\theta\in {\Sym}_m .
\end{equation}
In addition, for the case of $m=2$, let ${\Deq}_2:={\Sym}_2.$
We remark that a set which is defined by a similar
but different condition
is considered in \cite[Section 5.3]{ranking2}.
Nevertheless, we need the introduction of ${\Deq}_m$ above to approach
Theorem \ref{teiri1} below.
An important property satisfied by ${\Deq}_m$ is the closure property
$\widetilde{{\Deq}_m}={\Deq}_m$ with respect to the shift,
which will be presented as Proposition \ref{Prop2} in Section
\ref{SSSProp2}.

Here we state our main results, Theorems \ref{teiri1} and \ref{teiri2}.
Their proofs in this paper do not use 
the properties of ${\Sos}_m$, ${\Sinv}_m$ and ${\pNAP}_m$ in the previous section. In other words,   
they are formally independent of {\mysos} permutations and the inverses of {\mysos} permutations.

\begin{thm}\label{teiri1}
  Suppose that $m \ge 2$. Then
$${\Rcr}_m= {\Req}_m \ \text{\ and\ }\  {\Req}_m\subset {\Deq}_m.$$
\end{thm}

\begin{thm}\label{teiri2}
Suppose that $m \ge 2$. Then 
$$|{\Rcr}_m|
\le \sum_{k=1}^m \phi(k)
\ \text{\ and\ }\ 
|{\Deq}_m|
=|\widetilde{{\Deq}_m}|
\le m\sum_{k=1}^{m-1}\phi(k).$$
\end{thm}

By Theorems \ref{Known1}, \ref{Known2} and \ref{Known3}
in the previous section,
Theorems \ref{teiri1} and \ref{teiri2}
immediately produce the following corollary.

\begin{cor}\label{Cor1}
Suppose that $m\ge 2$. Then
$${\Sinv}_m={\Rcr}_m={\Req}_m \quad \text{and} \quad
\widetilde{{\Sinv}_m}={\Deq}_m=\widetilde{{\Rcr}_m}=\widetilde{{\Req}_m}.$$
In particular,
the set ${\Sinv}_m$ of the inverses of {\mysos} permutations of degree $m$ coincides with the set ${\Rcr}_m$,
and 
the set $\widetilde{{\Sinv}_m}$ of the inverses of permutations of {\mysos}-type of degree $m$ is the set ${\Deq}_m$.
\end{cor}

Note that Corollary \ref{Cor1} in particular asserts the converse of Theorem \ref{Known1}. Therefore now we have:
  For a permutation $\theta \in {\Sym}_m,$ satisfying
  the congruential recurrence \eqref{knowneq1} is equivalent
  to $\sigma := \theta^{-1}$ being a {\mysos} permutation
  which is characterized
  by Ineq.~\eqref{Sosnojyouken1} for some real number $\alpha$.

Moreover, Theorem \ref{Known5} and Corollary \ref{Cor1}
allow one to
deduce
the following directly.

\begin{cor}\label{Cor2}
  Suppose that $m\ge 2$. Then the set ${\pNAP}_m$ is the same as the set $\widetilde{{\Sinv}_m}$
of
  the inverses
  of permutations of {\mysos}-type of degree $m$:
$${\pNAP}_m=\widetilde{{\Sinv}_m}.
$$
\end{cor}
In other words, all permutation which forms a mod-$m$ congruential 
$m$-term
quasi-progression of diameter $1$
is the inverse of a {\mysos}-type permutation.

In the rest of this section, we prove Theorems \ref{teiri1} and \ref{teiri2}.
Our proofs are self-contained.

\subsection{Proof of Theorem \ref{teiri1}}
In this subsection, 
we show our proof of 
Theorem \ref{teiri1} by splitting it to the first part ${\Rcr}_m={\Req}_m$ and the second part ${\Req}_m\subset {\Deq}_m$.

\subsubsection{Proof of ${\Rcr}_m={\Req}_m$}
\begin{proof}
It is clear that the inclusion relationship ${\Req}_m\subset {\Rcr}_m$ holds by the definitions.
We will show that the inverted relationship ${\Rcr}_m\subset {\Req}_m$ holds.

Let $\theta$ be in ${\Rcr}_m$. Then there exists an integer $k_{\theta,i}$ which satisfies
that
\begin{equation}\label{kitukitousiki}
\theta(i+1)-\theta(i)-\theta(1)+\mychi(\theta(m)\le \theta(i)) = k_{\theta,i} m
\end{equation}
for $i\in [m-1]$.
We divide the LHS of Eq.~\eqref{kitukitousiki} into two parts.
Each part satisfies that
\begin{equation}\label{hutatunohutousiki}
1-m \le \theta(i+1)-\theta(i)\le m-1
\ \text{\ and\ }\
0\le \theta(1)-\mychi(\theta(m)\le\theta(i))\le m
\end{equation}
because 
$\theta(1),\theta(i)\text{\ and\ }\theta(i+1)\in [m]$.

Thus, the RHS of Eq.~\eqref{kitukitousiki} satisfies that $1-m -m \le k_{\theta,i} m \le m-1-0$,
that is, 
$$-2+\frac{1}{m} \le k_{\theta,i} \le 1-\frac{1}{m}$$
which show  bounds on $k_{\theta,i}$.
It follows that the integer $k_{\theta,i}$ must be either $-1$ or $0$.
We consider each case as follows.

When  
$k_{\theta,i}=-1$, 
Eq.~\eqref{kitukitousiki} implies that
$$
\theta(i+1)-\theta(i)=\theta(1)-\mychi(\theta(m)\le \theta(i))-m
$$
whose
RHS
is at most $0$, because it satisfies
$
\theta(1)-\mychi(\theta(m)\le \theta(i))-m\le m-m
$
by the second inequalities of \eqref{hutatunohutousiki}.
By $\theta(i+1)\not=\theta(i)$, 
we obtain $\theta(i+1)-\theta(i)\le -1$, that is, $\mychi(\theta(i)\le \theta(i+1))=0$,
which implies $k_{\theta,i}=\mychi(\theta(i)\le \theta(i+1))-1$
in this case of $k_{\theta,i}=-1$.

The case of $k_{\theta,i}=0$
 is similar:
Eq.~\eqref{kitukitousiki} means that
$
\theta(i+1)-\theta(i)=\theta(1)-\mychi(\theta(m)\le \theta(i))
$, whose RHS
is at least $0$, because of 
the second
  inequalities  
of \eqref{hutatunohutousiki}.
By $\theta(i+1)\not=\theta(i)$, 
we have $\theta(i+1)-\theta(i)\ge 1$, that is,
$\mychi(\theta(i)\le \theta(i+1))=1$
and then
$k_{\theta,i}=\mychi(\theta(i)\le \theta(i+1))-1$
in this case of $k_{\theta,i}=0$.

In both cases,
$$
\theta(i+1)-\theta(i)=\theta(1)-\mychi(\theta(m)\le \theta(i))+ m\left(\mychi(\theta(i)\le \theta(i+1))-1\right)
$$
holds. Therefore we conclude that $\theta\in {\Req}_m$ which implies ${\Rcr}_m\subset {\Req}_m$.
\end{proof}

\subsubsection{Proof of ${\Req}_m\subset {\Deq}_m$}
\begin{proof}
Let $\theta$ be in ${\Req}_m$. By the definition of ${\Req}_m$, 
\begin{equation}\label{meidai2notousiki1}
\theta(i+1)-\theta(i)+\mychi(\theta(m)\le \theta(i))-m \left(\mychi(\theta(i)\le \theta(i+1)) -1\right)= \theta(1) 
\end{equation}
holds for $i\in [m-1]$. In particular,
by the special case
$$
\theta(2)-\theta(1)+\mychi(\theta(m)\le \theta(1))-m \left(\mychi(\theta(1)\le \theta(2)) -1\right)= \theta(1)
$$
for $i=1$, we have
\begin{multline}
\theta(i+1)-\theta(i)+\mychi(\theta(m)\le \theta(i))-\mychi(\theta(i)\le \theta(i+1))-(m-1) \mychi(\theta(i)\le \theta(i+1)) \\
= \theta(2)-\theta(1)+\mychi(\theta(m)\le \theta(1))-\mychi(\theta(1)\le \theta(2))-(m-1) \mychi(\theta(1)\le \theta(2)).
\end{multline}
By using the notation $\Delta_\theta$
of Eq.~\eqref{Deltateigisiki}, the last equality is
equivalent to
$$\Delta_\theta(i)-\mychi(\theta(i)\le \theta(i+1))+\mychi(\theta(1)\le \theta(i+1))
=\Delta_\theta(1)$$
for $i\in [m-1]$.
Now we claim that
\begin{equation}\label{meidai2nokureemu}
\mychi(\theta(i)\le \theta(i+1))=\mychi(\theta(1)\le \theta(i+1))
\end{equation}
holds for $i\in [m-1]$.
If our claim, Eq.~\eqref{meidai2nokureemu}, is valid, then we have 
$\Delta_\theta(i)=\Delta_\theta(1)$
or equivalently
$\Delta_\theta(i+1)=\Delta_\theta(i)$ for $i\in [m-2]$,
that is,
$\theta\in {\Deq}_m$ which implies ${\Req}_m\subset {\Deq}_m$.

Let us show the validity of our claim, that is, Eq.~\eqref{meidai2nokureemu} holds for $i\in [m-1]$ as follows.
For $i\in [m-1]$, Eq.~\eqref{meidai2notousiki1} is
equivalent to  
\begin{equation}\label{meidai2notousiki3}
\theta(i+1)-\theta(1)=
\theta(i)-\mychi(\theta(m)\le \theta(i))
+m \left(\mychi(\theta(i)\le \theta(i+1)) -1\right).
\end{equation}
Because $\theta\in {\Sym}_m$ and $\mychi(\theta(m)\le \theta(i))\in\{0,1\}$,
we have $0\le \theta(i)-\mychi(\theta(m)\le \theta(i)) \le m$.
If $\theta(i)-\mychi(\theta(m)\le \theta(i))=0$, then 
both $\theta(i)=1$ and $\mychi(\theta(m)\le \theta(i))=1$ hold.
That is, $\theta(i)=1$ and $\theta(m)\le \theta(i)$. 
On the other hand, because
$\theta(m)\not=\theta(i)$ for $i\in [m-1]$, it is impossible that $\theta(m)\le \theta(i)$ for $\theta\in {\Sym}_m$.
Therefore $1\le \theta(i)-\mychi(\theta(m)\le \theta(i))$.
Similarly, if 
$\theta(i)-\mychi(\theta(m)\le \theta(i))=m$, then 
both 
$\theta(i)=m$ and $\mychi(\theta(m)\le \theta(i))=0$ hold, that is,
$\theta(i)=m$ and $\theta(m)> \theta(i)$.
But, by $\theta(m)\not=\theta(i)$ for $i\in [m-1]$, it is impossible that 
$\theta(m)> \theta(i)$ for $\theta\in {\Sym}_m$.
Then $\theta(i)-\mychi(\theta(m)\le \theta(i))\le m-1$.

Consequently, the inequalities 
\begin{equation}\label{meidai2nohutousiki3}
1\le \theta(i)-\mychi(\theta(m)\le \theta(i))\le m-1
\end{equation}
hold for $i\in [m-1]$.
Since $\mychi(\theta(i)\le \theta(i+1))$ is $0$ or $1$, 
we have to consider each case as follows.
When
$\mychi(\theta(i)\le \theta(i+1))=1$,
Eq.~\eqref{meidai2notousiki3} means that
$\theta(i+1)-\theta(1)=
\theta(i)-\mychi(\theta(m)\le \theta(i))
$.
Then
from
the inequalities \eqref{meidai2nohutousiki3}, it follows that
$\theta(i+1)-\theta(1)\ge 1$,
that is,
$\mychi(\theta(1)\le \theta(i+1))=1$.
When
$\mychi(\theta(i)\le \theta(i+1))=0$
is the case,
Eq.~\eqref{meidai2notousiki3} means that
$\theta(i+1)-\theta(1)=
\theta(i)-\mychi(\theta(m)\le \theta(i))-m
$
and from
the inequalities \eqref{meidai2nohutousiki3}, it follows that
$\theta(i+1)-\theta(1)\le -1$, that is,
$\mychi(\theta(1)\le \theta(i+1))=0$.

In either case, Eq.~\eqref{meidai2nokureemu} holds for $i\in [m-1]$. Therefore our claim is valid.
\end{proof}

\subsection{Some preliminaries and proof of Theorem \ref{teiri2}}
Our proof of Theorem \ref{teiri2} requires some preliminaries.

\subsubsection{An embedding of ${\Req}_{m-1}$ into ${\Deq}_m$}
We suppose that $m\ge 2$.
It is common to embed the set of permutations of degree $m-1$ into the set of permutations of degree $m$ having one particular fixed
point, say 1.
Here we introduce an explicit notation for such an embedding and its inverse.
Let ${\Sym}_m^1$ denote
$${\Sym}_m^1:=\{\theta\in {\Sym}_m\ : \ \theta(1)=1\}.
$$
We introduce the following map $\myPhi_m:{\Sym}_m^1\to {\Sym}_{m-1}$.
For $\theta\in
\Sym^1_m  
$,
$\myPhi_m(\theta) \in \Sym_{m-1}$ is defined by  
$$\myPhi_m(\theta)(i):=\theta(i+1)-1 \ \text{for}\ i\in [m-1].
$$
It is easy to check that $\myPhi_m:{\Sym}_m^1\to {\Sym}_{m-1}$ is well-defined and bijective
where the inverse $\myPhi_m^{-1}:{\Sym}_{m-1}\to {\Sym}_m^1$
satisfies, for given $\pi\in {\Sym}_{m-1}$, that
$$\myPhi_m^{-1}(\pi)(i):=\begin{cases}
1 & \ \text{if}\ i=1 ,\\
\pi(i-1)+1 & \ \text{if}\ i \ge 2 \\
\end{cases}\quad
\text{for $i\in [m]$}.
$$

The following Proposition \ref{Prop1} is similar to \cite[Theorem 37]{ranking2}.
We give its proof as below.

\begin{prp}\label{Prop1}
Suppose that $m\ge 3$.
Then
${\Req}_{m-1}$ is the bijective image of ${\Sym}_m^1\cap {\Deq}_m$ by $\myPhi_m$
:
$$\myPhi_m({\Sym}_m^1\cap {\Deq}_m)={\Req}_{m-1}.
$$
\end{prp}

\begin{proof}
The bijectivity
as the map $\myPhi_m:{\Sym}_m^1\to {\Sym}_{m-1}$
has been mentioned.
Below, the equality of the image and ${\Req}_{m-1}$ is shown.

First, we show that $\myPhi_m({\Sym}_m^1\cap {\Deq}_m)\supset {\Req}_{m-1}$.
It is enough to show that ${\Sym}_m^1\cap {\Deq}_m\supset \myPhi_m^{-1}({\Req}_{m-1})$
by using the inverse map $\myPhi_m^{-1}:{\Sym}_{m-1}\to {\Sym}_m^1$ of $\myPhi_m$.
Let $\pi$ be in ${\Req}_{m-1}$. Then it satisfies that 
\begin{equation}\label{itibanmenotousiki}
\pi(i+1)-\pi(i)=\pi(1)-\mychi(\pi(m-1)\le \pi(i))+(m-1)\left(\mychi(\pi(i)\le \pi(i+1))-1\right)
\end{equation}
for $i \in [m-2]$.
That is,
\begin{equation}\label{nibanmenotousiki}
\pi(i)-\pi(i-1)=\pi(1)-\mychi(\pi(m-1)\le \pi(i-1))+(m-1)\left(\mychi(\pi(i-1)\le \pi(i))-1\right)
\end{equation}
for $i \in [m-1]$ with $i\not=1$.
We put  $\theta:=\myPhi^{-1}_m(\pi)$. Since $\theta(i)-1=\pi(i-1)$ for $i\in [m-1]$ with $i\not=1$,
Eq.~\eqref{nibanmenotousiki} means that
\begin{equation}\label{sanbanmenotousiki}
\theta(i+1)-\theta(i)=\theta(2)-1-\mychi(\theta(m)\le \theta(i))+(m-1)\left(\mychi(\theta(i)\le \theta(i+1))-1\right).
\end{equation}
To ease the rest of our argument, let us rewrite it as 
\begin{equation}\label{meidai3notochuusiki1}
\theta(i+1)-\theta(i)
+\mychi(\theta(m)\le \theta(i))
-1
-(m-1)\mychi(\theta(i)\le \theta(i+1))
=
\theta(2)-1
+0
-1
-(m-1).
\end{equation}
By $i+1\ge 2$ and by $\theta(1)=1$, $\theta(i+1)\ge 2$ etc., we have
$\mychi(\theta(1)\le \theta(i+1))=1$,
$\mychi(\theta(1)\le \theta(2))=1$,
$\mychi(\theta(m)\le \theta(1))=0$,
$\mychi(\theta(1)\le \theta(2))=1$.
Therefore Eq.~\eqref{meidai3notochuusiki1} is the same as 
\begin{multline}\label{hinsyututousiki}
\theta(i+1)-\theta(i)
+\mychi(\theta(m)\le \theta(i))
-\mychi(\theta(1)\le \theta(i+1))
-(m-1)\mychi(\theta(i)\le \theta(i+1))\\
=
\theta(2)-\theta(1)
+\mychi(\theta(m)\le \theta(1))
-\mychi(\theta(1)\le \theta(2))
-(m-1)\mychi(\theta(1)\le \theta(2)),
\end{multline}
which,
with Eq.~\eqref{Deltateigisiki},
 shows that $\theta=\myPhi^{-1}_m(\pi)$ satisfies
$\Delta_\theta(i)=\Delta_\theta(1)$
for $i\in [m-1]$.
Thus, it follows that 
$\Delta_\theta(i+1)=\Delta_\theta(i)$
for $i\in [m-2]$ and then
$
\theta\in 
{\Deq}_m,
$
which, with $\theta(1)=1,$ implies that
${\Sym}_m^1\cap {\Deq}_m\supset \myPhi_m^{-1}({\Req}_{m-1})$.

We now show
the reverse inclusion
$\myPhi_m({\Sym}_m^1\cap {\Deq}_m)\subset {\Req}_{m-1}$.
The argument below
simply follows the path of the above argument backward.
Let $\theta\in {\Sym}_m^1\cap {\Deq}_m$, then $\theta$ satisfies 
$\Delta_\theta(i+1)=\Delta_\theta(i)$
for $i\in [m-2]$, that is, 
$\Delta_\theta(i)=\Delta_\theta(1)$
for $i\in [m-1]$ with 
$\theta(1)=1$.
So $\theta$ satisfies 
Eq.~\eqref{hinsyututousiki} with $\theta(1)=1$.
Since  $\theta\in {\Sym}_m$ and
$\theta(1)=1$,
we have
$\theta(i+1)\ge 2$ for $i+1\ge 2$.
From the facts that 
$\mychi(\theta(1)\le \theta(i+1))=1$,
$\mychi(\theta(m)\le \theta(1))=0$, etc., 
we have Eq.~\eqref{meidai3notochuusiki1}.
Thus,  
$\theta \in {\Sym}_m^1\cap {\Deq}_m$ satisfies 
Eq.~\eqref{sanbanmenotousiki} for $i\in [m-1]$ and $\theta(1)=1$.
Put $\pi:=\myPhi_m(\theta)$, then $\pi(i)=\theta(i+1)-1$ for $i\in [m-1]$.
Therefore Eq.~\eqref{sanbanmenotousiki} implies Eq.~\eqref{nibanmenotousiki}, i.e., 
Eq.~\eqref{itibanmenotousiki} for $i\in [m-2]$, 
which implies $\pi\in {\Req}_{m-1}$.
\end{proof}

\subsubsection{The closure property of ${\Deq}_m$ with respect to the shift}\label{SSSProp2}
Our proof of Theorem \ref{teiri2} will use the property that ${\Deq}_m$ is closed about the shift-equivalence.
In \cite[Proposition 35]{ranking2},
similar properties are considered in various generalized settings.
We give a proof of what we need as follows. 
Suppose that $m\ge 2$.

For
a given permutation $\theta\in {\Sym}_m$,
we put the number of
its ascents as  
$$A_\theta:=\sum_{j=1}^{m-1}\mychi(\theta(j)\le\theta(j+1)),
$$
then use it to define a set
$$
{\Deq}_m':=\{
\theta\in{\Sym}_m\ : \ 
\Delta_\theta(i)+A_\theta=0\ \text{for}\ i\in[m-1]
\}
$$
 of permutations.
The condition 
$\Delta_\theta(i)=-A_\theta$
is a special case of the first equality in \cite[Proposition 21]{ranking2}.
The following
Lemma \ref{Lem1} states that 
  this condition
  is in fact
  equivalent to
  (though sounds stricter than)
  the defining condition of ${\Deq}_m$.
\begin{lem}\label{Lem1}
Suppose that $m\ge 3$. Then ${\Deq}_m'={\Deq}_m$.
\end{lem}
\begin{proof}
Let $\theta$ be in $ {\Deq}_m'$. 
Then
$\Delta_\theta(i+1)=
-A_\theta=
\Delta_\theta(i)$
for $i\in [m-2]$, that is, $\theta\in{\Deq}_m$ which implies ${\Deq}_m'\subset{\Deq}_m$.
Let us show ${\Deq}_m'\supset{\Deq}_m$.
In the general case of $\theta\in {\Sym}_m$,
$\sum_{i=1}^{m-1}\mychi(\theta(m)> \theta(i))$
is the number of $\theta(1),\ldots,\theta(m-1)$ that is less than $\theta(m)$.
It is equal to $\theta(m)-1$.
Similarly, 
$\sum_{i=1}^{m-1}\mychi(\theta(1)> \theta(i+1))$ is equal to $\theta(1)-1$.
Hence
\begin{multline*}
\sum_{i=1}^{m-1}\Delta_\theta(i)=
\theta(m)-\theta(1)
+\sum_{i=1}^{m-1}\mychi(\theta(m)\le \theta(i))
-\sum_{i=1}^{m-1}\mychi(\theta(1)\le \theta(i+1))
-(m-1)\sum_{i=1}^{m-1}\mychi(\theta(i)\le \theta(i+1))\\
=\theta(m)-\theta(1)
+\sum_{i=1}^{m-1}\left(1-\mychi(\theta(m)> \theta(i))\right)
-\sum_{i=1}^{m-1}\left(1-\mychi(\theta(1)>\theta(i+1))\right)
-(m-1)\sum_{i=1}^{m-1}\mychi(\theta(i)\le \theta(i+1))\\
=-(m-1)\sum_{i=1}^{m-1}\mychi(\theta(i)\le \theta(i+1))=-(m-1)A_\theta .
\end{multline*}
Let $\theta$ be in ${\Deq}_m$.
Because $\Delta_\theta(i)=\Delta_\theta(1)$ for $i\in [m-1]$, we have
$\sum_{i=1}^{m-1}\Delta_\theta(i)=(m-1)\Delta_\theta(1)$.
It follows that $\Delta_\theta(1)=-A_\theta$, that is, $\Delta_\theta(i)=-A_\theta$ for $i\in [m-1]$,
which implies $\theta\in{\Deq}_m'$.
\end{proof}

  The above alternative definition of ${\Deq}_m$ helps proving its
  closure property with respect to the shift.
We give its proof as below.

\begin{prp}\label{Prop2}
Suppose that $m\ge 3$. If 
$\theta'\in {\Sym}_m$ satisfies that
$\theta'\sim \theta$ for some $\theta\in {\Deq}_m$,
then $\theta'\in {\Deq}_m$. That is, 
$\widetilde{{\Deq}_m}={\Deq}_m$.
\end{prp}

\begin{proof}
It is obvious that $\widetilde{{\Deq}_m}\supset{\Deq}_m$.
We prove that 
$\widetilde{{\Deq}_m}\subset{\Deq}_m$.

For $\theta\in {\Sym}_m$,
  let $\theta' = \lambda(\theta)$ be the {\it shift-by-one\/} of
  $\theta$ defined in Eq.~\eqref{sihuto1teigisiki}, then it satisfies that
$$\theta'(i)=\begin{cases}
\theta(i)+1 & \ \text{if}\ \theta(i)\le m-1 , \\
1           & \ \text{if}\ \theta(i)=m
\end{cases}\ \ \text{for}\ i\in [m] .
$$
  To prove $\widetilde{{\Deq}_m}\subset{\Deq}_m$, it is enough to
  show that $\widetilde{\{ \theta \}} \subset{\Deq}_m$
  for each $\theta \in {\Deq}_m$.
  Since $\widetilde{\{ \theta \}} = \langle \lambda \rangle \cdot \theta$,
  showing that $\theta' = \lambda(\theta) \in {\Deq}_m$ suffices.

First of all, we show that
\begin{equation}\label{marutto}
A_\theta+\mychi(\theta(m)\le \theta(1))=A_{\theta'}+\mychi(\theta'(m)\le \theta'(1))
\end{equation}
for $\theta \in {\Sym}_m$.
We introduce
the
shorthand
$\ryakuki{j}:=\supermod{m}(j)$
for integers $j$.
Then
$$\ryakuki{j}=\begin{cases}
\Mod_m(j) & \ \text{if}\ j\not\equiv 0 \mod m , \\
m & \ \text{if}\ j\equiv 0 \mod m .
\end{cases}$$
By using this notation, the LHS and the RHS of Eq.~\eqref{marutto} are written as
\begin{equation}\label{marutto2}
\sum_{j\in [m]} \mychi(\theta(\ryakuki{j})\le\theta(\ryakuki{j+1})),
\quad
\sum_{j\in [m]} \mychi(\theta'(\ryakuki{j})\le\theta'(\ryakuki{j+1}))
\end{equation}
respectively.
For the validity of Eq.~\eqref{marutto},
let us show that both values of \eqref{marutto2} are equal as follows.

We put $k :=\theta^{-1}(m)$, i.e.,
$k$ satisfies $k\in [m]$ and $\theta(k)=m$.
  Using $k$, the first sum in \eqref{marutto2} is split to
  3 parts, $j = \ryakuki{k-1}$, $j = k$ and others, namely
  $$
    \mychi(\theta(\ryakuki{k-1})\le\theta(\ryakuki{k}))
   + \mychi(\theta(\ryakuki{k})\le\theta(\ryakuki{k+1}))
   + \sum_{j\in [m] \setminus \{k, \ryakuki{k-1} \}}
  \mychi(\theta(\ryakuki{j})\le\theta(\ryakuki{j+1})).
  $$
  The second sum for $\theta'$ in \eqref{marutto2}
  also allows the same decomposition.
  In evaluating the first and the second parts
  with respect to $\theta$, since
both $\theta(\ryakuki{k-1})$ and $\theta(\ryakuki{k+1})$
are different values from $\theta(k)$,
they are less than
$m = \theta(k)$.
Similarly,
for the corresponding parts with respect to $\theta'$, 
both $\theta'(\ryakuki{k-1})$ and $\theta'(\ryakuki{k+1})$
are greater than
$1 = \theta'(k)$.
Then, for $\theta$,
  the sum
$
\mychi(\theta(\ryakuki{k-1})\le \theta(\ryakuki{k}))
+\mychi(\theta(\ryakuki{k})\le \theta(\ryakuki{k+1}))
$
of the first and the second parts
is equal to
$\mychi(\theta(\ryakuki{k-1})\le m)
+\mychi(m\le \theta(\ryakuki{k+1}))
$.
For $\theta'$,
  the sum 
$
\mychi(\theta'(\ryakuki{k-1})\le \theta'(\ryakuki{k}))
+\mychi(\theta'(\ryakuki{k})\le \theta'(\ryakuki{k+1}))
$
of the corresponding two terms
is equal to
$
\mychi(\theta'(\ryakuki{k-1})\le 1)
+\mychi(1\le \theta'(\ryakuki{k+1}))
$. 
All of these are equal to $1$.

For the third part,
neither $\ryakuki{j}$ nor $\ryakuki{j+1}$ is $k$.
Thus, for this part with respect to $\theta'$, each term 
$\mychi(\theta'(\ryakuki{j})\le \theta'(\ryakuki{j+1}))$
is the same as
$\mychi(\theta(\ryakuki{j})+1\le \theta(\ryakuki{j+1})+1)$
which is equal to
$\mychi(\theta(\ryakuki{j})\le \theta(\ryakuki{j+1}))$.
  Therefore we conclude that both values of \eqref{marutto2} are equal, 
that is, Eq.~\eqref{marutto} is valid.

By Lemma \ref{Lem1} and Eq.~\eqref{marutto}, 
for  the validity of ``$\theta\in {\Deq}_m$ implies $\theta'\in{\Deq}_m$'',
it is enough 
to show
that
$$
\Delta_\theta(i)
-\mychi(\theta(m)\le \theta(1))=
\Delta_{\theta'}(i)-\mychi(\theta'(m)\le \theta'(1))
$$
or an equivalent equality, in the expanded and rearranged form,
\begin{multline}\label{closedproofeq}
\left(\theta(i+1)-\theta'(i+1)\right)
-\left(\theta(i)-\theta'(i)\right)
-(m-1)\left(\mychi(\theta(i)\le \theta(i+1))-\mychi(\theta'(i)\le \theta'(i+1))\right)\\
=\left(\mychi(\theta(m)\le \theta(1))-\mychi(\theta'(m)\le \theta'(1))\right)\\
-\left(\mychi(\theta(m)\le \theta(i))-\mychi(\theta'(m)\le \theta'(i))\right)
+\left(\mychi(\theta(1)\le \theta(i+1))-\mychi(\theta'(1)\le \theta'(i+1))\right)
\end{multline}
holds for $i\in [m-1]$.

Let us consider it for all possible cases separately.
First, we check what kind of cases can occur.
For each $i\in [m-1]$, we put $Adj_i:=(\mychi(\theta(i)=m),\mychi(\theta(i+1)=m))$.
Then
there are three cases of $Adj_i$: $(0,0)$, $(1,0)$ and $(0,1)$.
We also put $HT:=(\mychi(\theta(1)=m),\mychi(\theta(m)=m))$.
Then there are also three cases of $HT$: $(0,0)$, $(1,0)$ and $(0,1)$.
Therefore there are 9 possible cases, from Case A to Case I, in the following table
\begin{center}
\begin{tabular}{c|ccc}
$HT$
$\backslash$
$Adj_i$
& $(0,0)$ & $(1,0)$ & $(0,1)$ \\ \hline
$(0,0)$ & A & B & C \\
$(1,0)$ & D & E & F \\
$(0,1)$ & G & H & I \\
\end{tabular}\end{center}
but Case F and Case H are impossible.
So, we consider the remaining 7 cases below.

Before considering these 7 cases, we will find the value of the LHS of Eq.~\eqref{closedproofeq} in each of the three possible cases of $Adj_i$.


\noindent
The case of $Adj_i=(0,0)$:
In this case, it follows that
$\theta'(i)=\theta(i)+1$ and
$\theta'(i+1)=\theta(i+1)+1$.
Then
the LHS of Eq.~\eqref{closedproofeq} is $0$.


\noindent
The case of $Adj_i=(1,0)$:
In this case, it follows that
$\theta(i)=m$,
$\theta'(i)=1$ and 
$\theta'(i+1)=\theta(i+1)+1$.
Then
the LHS of Eq.~\eqref{closedproofeq} is 
\begin{multline}
\left(\theta(i+1)-\theta'(i+1)\right)-\left(\theta(i)-\theta'(i)\right)-(m-1)\left(\mychi(\theta(i)\le \theta(i+1))-\mychi(\theta'(i)\le \theta'(i+1))\right)\\ 
=
-1
-\left(m-1\right)
-(m-1)\left(\mychi(m\le \theta(i+1))-\mychi(1\le \theta'(i+1))\right)
=
-1
-\left(m-1\right)
-(m-1)\left(0-1\right)
\end{multline}
which is equal to $-1$.


\noindent
The case of $Adj_i=(0,1)$:
In this case, it follows that
$\theta'(i)=\theta(i)+1$,
$\theta(i+1)=m$ and
$\theta'(i+1)=1$. 
Then
the LHS of Eq.~\eqref{closedproofeq} is
$$(m-1)-(-1)
-(m-1)\left(\mychi(\theta(i)\le m)-\mychi(\theta'(i)\le 1)\right)
=
(m-1)-(-1)
-(m-1)\left(1-0\right)
$$
which is equal to $1$.

Using the value for each of the three cases above, 
let us show the validity of Eq.~\eqref{closedproofeq} for 7 cases separately as follows.

 
\noindent
Case A:
In this case, it follows that
$\theta'(i)=\theta(i)+1$,
$\theta'(i+1)=\theta(i+1)+1$,
$\theta'(1)=\theta(1)+1$ and
$\theta'(m)=\theta(m)+1$.
Then
the RHS of 
Eq.~\eqref{closedproofeq} is 
\begin{multline}
\left(\mychi(\theta(m)\le \theta(1))-\mychi(\theta'(m)\le \theta'(1))\right)\\
-\left(\mychi(\theta(m)\le \theta(i))-\mychi(\theta'(m)\le \theta'(i))\right)
+\left(\mychi(\theta(1)\le \theta(i+1))-\mychi(\theta'(1)\le \theta'(i+1))\right)\\
=
0-0+0
\end{multline}
which is equal to the LHS of Eq.~\eqref{closedproofeq}.
Therefore Eq.~\eqref{closedproofeq} holds.


\noindent
Case B:
In this case, it follows that
$\theta(i)=m$,
$\theta'(i)=1$,
$\theta'(i+1)=\theta(i+1)+1$,
$\theta'(1)=\theta(1)+1$,
$\theta'(m)=\theta(m)+1$
and that
$\theta(m)$ is not $m$, i.e., $\theta'(m)$ is not $1$.
Then the RHS of Eq.~\eqref{closedproofeq} is
$$0
-\left(\mychi(\theta(m)\le m)-\mychi(\theta'(m)\le 1)\right)
+0
=0
-\left(1-0\right)
+0
$$
which is equal to the LHS of Eq.~\eqref{closedproofeq}.


\noindent
Case C:
In this case, it follows that
$\theta'(i)=\theta(i)+1$,
$\theta(i+1)=m$,
$\theta'(i+1)=1$,
$\theta'(1)=\theta(1)+1$,
$\theta'(m)=\theta(m)+1$
and that 
$\theta(1)$ is not $m$, i.e.,
$\theta'(1)$ is not $1$.
Then the RHS of Eq.~\eqref{closedproofeq} is 
$$0
-0
+\left(\mychi(\theta(1)\le m)-\mychi(\theta'(1)\le 1)\right)\\
=(1-0)
$$
which is equal to the LHS of Eq.~\eqref{closedproofeq}.


\noindent
Case D:
In this case, it follows that
$\theta'(i)=\theta(i)+1$,
$\theta'(i+1)=\theta(i+1)+1$,
$\theta(1)=m$,
$\theta'(1)=1$,
$\theta'(m)=\theta(m)+1$
and that 
$\theta(i+1)$ is not $m$, i.e., 
$\theta'(i+1)$ is not $1$.
Similarly,
$\theta(m)$ is not $m$, i.e., 
$\theta'(m)$ is not $1$.
Therefore the RHS of Eq.~\eqref{closedproofeq} is 
$$
\left(\mychi(\theta(m)\le m)-\mychi(\theta'(m)\le 1)\right)
-0
+\left(\mychi(m\le \theta(i+1))-\mychi(1\le \theta'(i+1))\right)
=(1-0)-0+(0-1)
$$
which is equal to the LHS of Eq.~\eqref{closedproofeq}.

 
\noindent
Case E:
In this case, it follows that
$\theta(i)=m$,
$\theta'(i)=1$,
$\theta'(i+1)=\theta(i+1)+1$,
$\theta(1)=m$,
$\theta'(1)=1$,
$\theta'(m)=\theta(m)+1$ and that
$\theta(i+1)$ is not $m$, i.e., $\theta'(i+1)$ is not $1$.
Similarly,
$\theta(m)$ is not $m$, i.e., $\theta'(m)$ is not $1$.
Therefore the RHS of Eq.~\eqref{closedproofeq} is

\begin{multline}
\left(\mychi(\theta(m)\le m)-\mychi(\theta'(m)\le 1)\right)
-\left(\mychi(\theta(m)\le m)-\mychi(\theta'(m)\le 1)\right)\\
+\left(\mychi(m\le \theta(i+1))-\mychi(1\le \theta'(i+1))\right)
=(1-0)-(1-0)+(0-1)
\end{multline}
which is equal to the LHS of Eq.~\eqref{closedproofeq}.

 
\noindent
Case G:
In this case, it follows that
$\theta'(i)=\theta(i)+1$,
$\theta'(i+1)=\theta(i+1)+1$,
$\theta'(1)=\theta(1)+1$,
$\theta(m)=m$,
$\theta'(m)=1$ and that
$\theta(1)$ is not $m$, i.e., $\theta'(1)$ is not $1$.
Similarly,
$\theta(i)$ is not $m$, i.e.,  $\theta'(i)$ is not $1$.
Therefore the RHS of Eq.~\eqref{closedproofeq} is 
\begin{multline}
\left(\mychi(m\le \theta(1))-\mychi(1\le \theta'(1))\right)
-\left(\mychi(m\le \theta(i))-\mychi(1\le \theta'(i))\right)\\
+\left(\mychi(\theta(1)\le \theta(i+1))-\mychi(\theta'(1)\le \theta'(i+1))\right)
=(0-1)-(0-1)+0
\end{multline}
which is equal to the LHS of Eq.~\eqref{closedproofeq}.


\noindent
Case I:
In this case, it follows that
$\theta'(i)=\theta(i)+1$,
$\theta(i+1)=m$,
$\theta'(i+1)=1$,
$\theta'(1)=\theta(1)+1$,
$\theta(m)=m$,
$\theta'(m)=1$ and that 
$\theta(i)$ is not $m$, i.e., $\theta'(i)$ is not $1$.
Similarly,
$\theta(1)$ is not $m$, i.e., $\theta'(1)$ is not $1$.
Therefore the RHS of Eq.~\eqref{closedproofeq} is
\begin{multline}
\left(\mychi(m\le \theta(1))-\mychi(1\le \theta'(1))\right)
-\left(\mychi(m\le \theta(i))-\mychi(1\le \theta'(i))\right)\\
+\left(\mychi(\theta(1)\le m)-\mychi(\theta'(1)\le 1)\right)
=(0-1)-(0-1)+(1-0)
\end{multline}
which equal to the LHS of Eq.~\eqref{closedproofeq}.

From the above arguments, Eq.~\eqref{closedproofeq} holds for all possible cases and for each $i \in [m-1]$.
The proof is completed.
\end{proof}

\subsubsection{A property of ${\Rcr}_m$}

  In this subsection, some key facts on the structure of the set
  ${\Rcr}_m$ are shown.

For any elements $a, b$ in the ring $\mathbb{Z}$ of integers, let
  \begin{equation}\label{eq:thetaabdef}
  \theta_{a, b}(i) := \supermodtwo{m}{a i + b} \text{\ for \ }i \in [m],
  \end{equation}
  which defines a map $\theta_{a, b}:[m] \rightarrow [m].$

\begin{prp}\label{PropT}
  Suppose that $m \geq 2$.
  Let
  $\Rcr_m^{\mathrm{L}, b}
  := \{\theta_{a, b} : a \in [m], \gcd(a,m) = 1 \}$
  for $b \in \{0, 1\}$. The following holds:
  \begin{itemize}
  \item[(i)] It holds that
    $\Rcr_m^{\mathrm{L}, 0} \cap \Rcr_m^{\mathrm{L}, 1} = \emptyset$,
    that $\Rcr_m^{\mathrm{L}, b} \subset  \Rcr_m$ for $b = 0, 1$,
    and that
    $|\Rcr_m^{\mathrm{L}, b}| = \phi(m)$ for $b = 0, 1$.
    The cardinality of the set
    $\Rcr_m^{-} := \Rcr_m \setminus \Rcr_m^{\mathrm{L}, 1}$ 
    is
    $|\Rcr_m^{-}| = |\Rcr_m| - \phi(m).$ 
  \item[(ii)]   Let $m \geq 3$.
    Suppose that
    $\theta, \theta' \in \Rcr_m$ satisfy $\theta \sim \theta'$
    and $\theta \neq \theta'$. Then, exactly one of
    $\{\theta, \theta'\}$ is in $\Rcr_m^{\mathrm{L},{0}}$ and
    the other is in $\Rcr_m^{\mathrm{L},{1}}$.
    Especially, $\theta \sim \theta'$ for
    $\theta, \theta' \in \Rcr_m^{-}$
    implies  $\theta = \theta'.$ 
  \end{itemize}

\end{prp}

\begin{proof}
\newcommand{\changedblue}[1]{\textcolor{blue}{#1}}
\newcommand{\changedgreen}[1]{\textcolor{teal}{#1}}
 (i)
    Well-known facts are that
    $\{a + m \mathbb{Z}: a \in [m], \gcd(a,m)=1\}$ 
    is a complete representative system
    of the invertible elements in the residue ring 
    $\mathbb{Z}/m\mathbb{Z}$, that
    $\theta_{a,0}\in \Sym_m$ if and only if $\gcd(a,m)=1$ and that
    $|\Rcr_m^{\mathrm{L}, 0}| = \phi(m)$.
    Since $\Rcr_m^{\mathrm{L}, 1}$ is the bijective image of
    $\Rcr_m^{\mathrm{L}, 0}$ by a shift, it has the same cardinality.
    Suppose that
    $a, a' \in \mathbb{Z}$ satisfy $\gcd(a,m)=\gcd(a',m)=1.$
    For $b, b' \in \mathbb{Z},$ the relation
    $\theta_{a,b} = \theta_{a', b'}$ implies
    $(a-a')i + (b-b') \equiv 0 \mod m$ for $i \in [m]$,
    then we have $b \equiv b' \mod m$ 
    by substituting $i$ with $m$,
    then $a \equiv a' \mod m$ by substituting $i$ with $1$. That is,
    $\theta_{a, b} = \theta_{a', b'}$ occurs for $a, a', \in [m],
    b, b' \in \{0\} \cup [m-1]$
    only if $(a,b)=(a',b')$. Thus, in particular,
    $\Rcr_m^{\mathrm{L}, 0} \cap \Rcr_m^{\mathrm{L}, 1} = \emptyset.$ 

    For all $\theta_{a,0} \in \Rcr_m^{\mathrm{L}, 0}, $ by
    $\supermodtwo{m}{\theta_{a,0}(m)}
    = \supermodtwo{m}{a m}   = \supermodtwo{m}{0} = m,$
    it follows that $\mychi(\theta_{a,0}(m) \leq \theta_{a,0}(i)) = 0\
    \text{\ for\ }i \in [m-1].$ 
    Thus, by
    \begin{equation}
    \theta_{a,0}(i+1) - \theta_{a,0}(i) \equiv
    a (i+1) - a i \equiv a 
    \equiv \theta_{a,0}(1) -
    \mychi(\theta_{a,0}(m) \leq \theta_{a,0}(i)) \mod m
    \ \text{\ for\ } i \in [m-1],
    \end{equation}
    we have $\theta_{a,0} \in \Rcr_m$, which implies
    $\Rcr_m^{\mathrm{L}, 0} \subset \Rcr_m$.

    And for $\theta_{a,1} \in \Rcr_m^{\mathrm{L}, 1}$, from
    $\theta_{a,1}(m) = \supermodtwo{m}{a m + 1} =  \supermodtwo{m}{1} = 1$,
    we have that
    $\mychi(\theta_{a,1}(m) \leq \theta_{a,1}(i)) = 1\
    (i \in [m-1])$ and that
    \begin{multline*}
    \theta_{a,1}(i+1) - \theta_{a,1}(i) \equiv
    a (i+1) + 1 - (a i + 1)\equiv a \\
    \equiv \theta_{a,1}(1) -
    \mychi(\theta_{a,1}(m) \leq \theta_{a,1}(i)) \mod m
    \ \text{\ for\ }  i \in [m-1],
    \end{multline*}
    which implies $\theta_{a,1} \in \Rcr_m$  and
    $\Rcr_m^{\mathrm{L}, 1} \subset \Rcr_m.$
    Thus,
    $|\Rcr_m^{-}| = |\Rcr_m| - |\Rcr_m^{\mathrm{L}, 1}| =
|\Rcr_m| - 
    \phi(m).$

(ii)
  Suppose that $\theta, \theta' \in \Rcr_m$ satisfy
  $\theta \sim \theta'$.
  Then
  by Eq.~\eqref{sihutodoutiteigino2},
  there exists $k \in \mathbb{Z}$ such that
\begin{equation}
  \theta'(i) \equiv \theta(i) + k \mod m \text{\ for\ } i \in [m].
\label{eq:simdef}
\end{equation}
By $\theta' \in \Rcr_m$, we have
\begin{equation}
 (\theta(i+1)+k) - (\theta(i)+k) \\ \equiv \theta'(1) - 
  \mychi(\theta'(m) \le \theta'(i)) \mod m 
  \text{\ for\ } i \in [m-1].
\end{equation}
Comparing it with the following congruence coming from $\theta \in \Rcr_m$,
\begin{equation}
\theta(i+1) - \theta(i)  \equiv \theta(1) - 
  \mychi(\theta(m) \le \theta(i)) \mod m 
  \text{\ for\ } i \in [m-1],
\label{eq:thetaeqv}
\end{equation}
we find that their LHS are the same, and we are led to
\begin{equation}
  \theta(1) - 
  \mychi(\theta(m) \le \theta(i))  \\
  \equiv
  \theta'(1) - 
  \mychi(\theta'(m) \le \theta'(i)) \mod m 
  \text{\ for\ } i \in [m-1].
\end{equation}
By rearranging terms, with recalling Eq.~\eqref{eq:simdef}, it follows that
\begin{equation}
  k \equiv \theta'(1) - \theta(1) 
  \equiv
  \mychi(\theta'(m) \le \theta'(i))
-  \mychi(\theta(m) \le \theta(i)) 
  \mod m \text{\ for\ } i \in [m-1].
\label{eq:Ceqv}
\end{equation}
The rightmost hand of the above takes only three values $0, 1, -1$. 
Thus, only the three cases $k \equiv 0, 1, -1 \mod m$ are possible.

When $k \equiv 0 \mod m$, Eq.~\eqref{eq:simdef} is
\begin{equation*}
  \theta'(i) \equiv \theta(i) \mod m \text{\ for\ } i \in [m],
\end{equation*}
then $\theta = \theta'$ by $\theta'(i), \theta(i) \in [m].$

When $k \equiv 1 \mod m,$ by Eq.~\eqref{eq:Ceqv},
\begin{equation*}
  1 \equiv 
  \mychi(\theta'(m) \le \theta'(i)) 
-  \mychi(\theta(m) \le \theta(i))
  \mod m  \text{\ for\ } i \in [m-1].
\end{equation*}
By $m \geq 3$ 
and the fact that the image of $\mychi$
is contained in $\{0, 1\}$, this is possible only if
$$
\mychi(\theta'(m) \le \theta'(i)) = 1 \ \text{and} \
\mychi(\theta(m) \le \theta(i)) = 0 \text{\ for\ } i \in [m-1].
$$
Then, from the latter half of the logical conjunction, with Eq.~\eqref{eq:thetaeqv},
$\theta(i+1)-\theta(i) \equiv \theta(1) \mod m
\text{\ for\ } i \in [m-1]$ follows. Put $a := \theta(1) \in [m]$.
Then we have that
$\theta(i) \equiv a i \mod m \text{\ for\ } i \in [m]$
and that $\gcd(a, m)=1$ by $\theta \in \Sym_m$.
Thus, $\theta = \theta_{a,0} \in \Rcr_m^{\mathrm{L}, 0}$.
Then, $\theta'(i) \equiv \theta(i) + 1 \mod m \text{\ for\ } i \in [m]$
implies $\theta' = \theta_{a,1}$.

When 
$k\equiv -1 \mod m$, swapping $\theta$ and $\theta'$ reduces
this case to the previous case $k \equiv 1 \mod m$. Thus we have
$\theta' = \theta_{a,0}, \theta = \theta_{a,1}$ by the same
argument.

At this point it has been shown that
one of $\theta, \theta'$ must belong to $\Rcr_m^{\mathrm{L},1}$
when
$\theta, \theta' \in \Rcr_m$ satisfy
$\theta \neq \theta', \theta \sim \theta'$.
Therefore, when
$\theta, \theta' \in \Rcr_m^{-}$  satisfy $\theta \sim \theta'$,
the equality
$\theta = \theta'$ must hold since none of
$\theta, \theta'$ belongs to $\Rcr_m^{\mathrm{L},1}$.
\end{proof}

\subsubsection{Proof of Theorem \ref{teiri2}}
We are now ready to provide the proof of Theorem \ref{teiri2} below.

\begin{proof}
By Theorem \ref{teiri1} and 
Proposition \ref{PropT}, 
there exists a set ${\Req}^-_m \subset {\Req}_m$
which satisfies the two conditions that 
``$|{\Req}^-_m|=|{\Req}_m|-\phi(m)$'' and ``$\theta\sim \theta'$ for $\theta,\theta'\in {\Req}^-_m$ implies $\theta=\theta'$.'' 
  From the latter condition, for each $\theta, \theta' \in {\Req}^-_m$,
  we have
  $\widetilde{\{\theta\}} \cap \widetilde{\{\theta'\}} = \emptyset$
  unless $\theta = \theta'$, since the existence of $\theta''$ in
  this intersection implies $\theta \sim \theta'' \sim \theta'$.
  Considering that the orbit $\widetilde{\{\theta\}}
= \langle \lambda \rangle \cdot \theta
  $
  for any $\theta \in \Sym_m$
  always contains exactly $m$ distinct elements corresponding to their values
  at 1, we have $|\widetilde{{\Req}^-_m}|=m|{\Req}^-_m|$
  from the disjointness. By a similar argument, with
  $\widetilde{{\Deq}_m} = {\Deq}_m$ of Proposition \ref{Prop2},
  we also have $|{\Deq}_m|=|\widetilde{{\Deq}_m}|=
  |\widetilde{{\Sym}_m^1\cap{\Deq}_m}|=m|{\Sym}_m^1\cap{\Deq}_m|$. 

  By Theorem \ref{teiri1},
we have ${\Req}^-_m \subset {\Deq}_m$. Thus,
taking the shift-closure and applying Proposition \ref{Prop2} again,
it follows that
$\widetilde{{\Req}^-_m } \subset \widetilde{{\Deq}_m} ={\Deq}_m$. In particular, we have
the inequality
$m|{\Req}^-_m|\le |{\Deq}_m|$ whose LHS is $m (|{\Req}_m|-\phi(m))$
from the former condition satisfied by ${\Req}^-_m$.

By Proposition \ref{Prop1} and by the injectivity of the map $\myPhi_m$,
we have 
$|{\Sym}_m^1\cap{\Deq}_m|=|\myPhi_m({\Sym}_m^1\cap{\Deq}_m)|
=|{\Req}_{m-1}|$.
Consequently, 
$m\left(|{\Req}_m|-\phi(m)\right) =m|{\Req}^-_m|\le m|{\Req}_{m-1}|$,
that is,
$|{\Req}_m|-|{\Req}_{m-1}| \le \phi(m)$.

Again by Theorem \ref{teiri1}, we have $|{\Rcr}_m|=|{\Req}_m|$.
It is easy to check that  ${\Req}_2={\Sym}_2$ and that
$|{\Req}_2|=2=\phi(1)+\phi(2)$, by a direct argument.
Therefore we conclude that 
$|{\Rcr}_m| \le \sum_{k=1}^m \phi(k)$
and that
$|{\Deq}_m|=m|{\Sym}_m^1\cap{\Deq}_m|=m|{\Req}_{m-1}| 
\le m\sum_{k=1}^{m-1} \phi(k)$.
\end{proof}

\section{An application}
\renewcommand{\myPhi}{\Psi}
\newcommand{\myPsi}{\Gamma}
\def \Sym {\frak{S}}
\def \Sos {\mathcal{S}}
\def \Sinv {\mathcal{S}^{*}}
\def \Rcr {\mathcal{V}}
\def \Rcrlike {V}
\def \Req {\mathcal{W}}
\def \pNAP {\mathcal{X}}
\def \Deq {\mathcal{Y}}
\def \CDS {\mathsf{CDS}}
\newcommand{\RefToThmXmVm}{Theorem B {}}
\newcommand{\RefToThmSmCard}{Theorem C {}}
\newcommand{\RefToThmSmClosCard}{Theorem D {}}
A key ingredient in the proof of Theorem \ref{teiri2} is
the use of the map $\myPhi_{m}$ which bijectively
connects the two sets $\Sym_m^1 \cap \Deq_{m}$ and
$\Rcr_{m-1}$ of permutations of {\it different\/} degrees.

In this section, we look at another building block
which connects $\Rcr_{m}$ and $\Sym_m^1 \cap \Deq_{m}$,
then combine it with $\myPhi_{m}$ to form a procedure
that lifts the elements of $\Rcr_{m-1}$ to $\Rcr_{m}$ in a
canonical manner. 
In this section we assume that $m \geq 3$ unless otherwise stated.
\subsection{Connecting $\Rcr_{m}$ and $\Sym_m^1 \cap \Deq_{m}$ by shift}
The building block we consider here is the map
\begin{equation}
  \myPsi_{m}: \Sym_{m} \ni \theta \mapsto
  \supermodtwo{m}{\theta(\cdot) - \theta(1) + 1}
   \in \Sym_{m}^{1}
\label{eq:psi}
\end{equation}
which applies a shift so that a permutations of degree $m$ is sent to
the permutation of the same degree having $1$ as a fixed point.
From the definition, it is obvious that $\myPsi_m$ preserves
the shift-equivalence class, i.e., $\theta \sim \myPsi_{m}(\theta)$.

Let us consider the restriction
$\myPsi_{m}|_{\Rcr_{m}}$. From the definition \eqref{eq:psi},
it holds that $\myPsi_{m}(\Rcr_{m})
\subset \widetilde{\Rcr_{m}} \cap \Sym_{m}^{1}.$
From Theorem \ref{teiri1},
$\widetilde{\Rcr_{m}} \subset \widetilde{\Deq}_m$ holds,
then by Proposition \ref{Prop2}, it follows that
$\widetilde{\Rcr_{m}} \subset \Deq_m$. Thus, $\myPsi_{m}|_{\Rcr_{m}}$
is a map
\begin{equation*}
  \myPsi_{m}|_{\Rcr_{m}}: \Rcr_{m} \rightarrow \Deq_m  \cap \Sym_{m}^{1}.
\end{equation*}

We note that
$\Deq_m \cap \Sym_{m}^{1}$ is a complete system
of the representatives for $\Deq_m /\sim, $ again by
Proposition \ref{Prop2} (i.e., for any $\theta \in \Deq_m$
there exists a unique $\theta' \in \Deq_m \cap \Sym_{m}^{1}$ such
that $\theta' \sim \theta$).

Then, let us consider what restriction does make $\myPsi_m$
injective.
As we have shown in Proposition \ref{PropT} (ii),
$\theta \sim \theta'$ for $\theta, \theta' \in \Rcr_{m}^{-}
= \Rcr_{m} \setminus \Rcr_{m}^{\mathrm{L}, 1}$ implies
$\theta = \theta'$. Since $\myPsi_m$ preserves the
equivalence class by $\sim$, $\myPsi_m(\theta) = \myPsi_m(\theta')$
implies $\theta \sim \theta'$. Combining these, the restriction
$\myPsi_{m}|_{\Rcr_{m}^{-}}$ is injective.

When the domain is restored to $\Rcr_{m}$, $\myPsi_{m}|_{\Rcr_{m}}$
is 
no longer injective. However, in  fact, the manner that a collision occurs is
quite well-controlled as shown below. Suppose that
$\theta, \theta' \in \Rcr_{m}, \theta \neq \theta'$ satisfy
$\myPsi_{m}(\theta) = \myPsi_{m}(\theta').$ Then
by the shift-equivalence preservation of $\myPsi_{m}$, we have
$\theta \sim \theta'$, however,
again by Proposition \ref{PropT} (ii), it
occurs if and only if one of $\theta, \theta'$
is in $\Rcr_{m}^{\mathrm{L},0}$
and the other is in $\Rcr_{m}^{\mathrm{L},1}$. In this case,
it is easy to see, by the shift-equivalence preservation and
$\myPsi_{m}(\Rcr_{m}) \subset \Sym_{m}^{1}$,
that
$\myPsi_{m}(\theta)(i)=\theta_{a,m-a+1}(i)$
for $\theta = \theta_{a,0} \in \Rcr_{m}^{\mathrm{L},0}$
and
$\theta = \theta_{a,1} \in \Rcr_{m}^{\mathrm{L},1}$,
where $\theta_{a,*}$ were defined in Eq.~\eqref{eq:thetaabdef}. 
On the other hand, if $\pi \in \Deq_m \cap \Sym_{m}^{1}$
is $\pi = \theta_{a,m-a+1}$ for some $a \in [m]$ with $\gcd(a,m) = 1$,
then it is clear $\pi = \myPsi_{m}(\theta_{a,0}) =
\myPsi_{m}(\theta_{a,1})$. Thus,
$(\myPsi_{m}|_{{\Rcr_m^{\mathrm{L}, 0} \sqcup \Rcr_m^{\mathrm{L}, 1}}})^{-1}(\{\pi\})
= \{\theta_{a,0}, \theta_{a,1} \}$ for $\pi \in \Deq_m \cap \Sym_{m}^{1}$
occurs if and only if
$\pi = \theta_{a,m-a+1}$ for some $a$ with $\gcd(a,m) = 1$.

For $\theta \in \Rcr_m^{\mathrm{L}, 0} \sqcup \Rcr_m^{\mathrm{L}, 1}$
and
$\theta' \in \Rcr_m \setminus (\Rcr_m^{\mathrm{L}, 0} \sqcup \Rcr_m^{\mathrm{L}, 1})$, $\myPsi_m(\theta)=\myPsi_m(\theta')$ is impossible
 as it implies $\theta \sim \theta'$ which is refuted by 
Proposition \ref{PropT} (ii) again.

We summarize these properties of $\myPsi_m$ and its restrictions
in the following lemma. 
For convenience, we introduce a notation for
  {\it the congruential difference set\/} for a sequence $\theta$ over $[m]$
  as
  \begin{equation}
    \CDS_m(\theta)
    := \{ \mathrm{Mod}_m(\theta(i+1)-\theta(i)) : i \in [m-1] \}.
  \end{equation}

\begin{lem}\label{lem:proj}
  Let $m \geq 3$.
  Let $\myPsi_{m}: \Sym_{m} \rightarrow \Sym_{m}^{1}$ be the map
  defined as \eqref{eq:psi}. Then, the following holds:
  \begin{itemize}
    \item[(i)] $\theta \sim \myPsi_m(\theta)$ for $\theta \in \Sym_m$.
    \item[(ii)] The restriction $\myPsi_{m}|_{\Rcr_m}$ is a map
      $ \Rcr_m \rightarrow \Deq_m \cap \Sym_{m}^{1}$.
    \item[(iii)] The restriction
      $\myPsi_{m}|_{\Rcr_m \setminus \Rcr_m^{\mathrm{L}, b}}$
      is an injection from $\Rcr_m \setminus \Rcr_m^{\mathrm{L}, b}$
      to $\Sym_m^1 \cap \Deq_m$ for each of $b = 0, 1$.
    \item[(iv)] The restriction
      $\myPsi_{m}|_{\Rcr_m^{\mathrm{L}, 0} \sqcup \Rcr_m^{\mathrm{L}, 1}}$
      is a $2:1$ surjection from
      $\Rcr_m^{\mathrm{L}, 0} \sqcup \Rcr_m^{\mathrm{L}, 1}$
      onto
      $\{ \theta_{a, m-a+1} : a \in [m], \gcd(a,m)=1\}$ whose
      cardinality is $\phi(m)$.
      And $(\myPsi_{m}|_{\Rcr_m^{\mathrm{L}, 0} \sqcup \Rcr_m^{\mathrm{L}, 1}})^{-1}(\{\theta_{a, m-a+1}\}) = \{\theta_{a,0}, \theta_{a,1}\}.$
    \item[(v)] $\myPsi_{m}(\Rcr_m^{\mathrm{L}, 0} \sqcup \Rcr_m^{\mathrm{L}, 1}) \cap \myPsi_{m}(\Rcr_m \setminus (\Rcr_m^{\mathrm{L}, 0} \sqcup \Rcr_m^{\mathrm{L}, 1})) = \emptyset.$
    \item[(vi)] The cardinality of the congruential
      difference set $\CDS_m(\myPsi_m(\theta))$
  for $\theta \in \Rcr_m$ is 1 if and only if
      $\theta \in \Rcr_m^{\mathrm{L}, 0} \sqcup \Rcr_m^{\mathrm{L}, 1}$.
  \end{itemize}
\end{lem}

\begin{proof}
  We already have shown (i)-(v). For (vi),
  ``if'' part easily follows from (iv). For ``only if'' part,
  suppose that
  a $\theta \in \Rcr_m$ has the congruential difference set 
  $\CDS_m(\myPsi_m(\theta))$
  of cardinality 1. If the set is $\{a\},$
  then $\myPsi_m(\theta(i)) = \supermodtwo{m}{a i + b} \ (i \in [m])$
  for some $b \in \mathbb{Z},$ however, then $\gcd(a,m) = 1$
  holds because otherwise $\myPsi_m(\theta)$ is not a permutation.
  Thus, we have $\theta \sim \myPsi_m(\theta) \sim \theta_{a,0} \sim
  \myPsi_m(\theta_{a,0})$
  from (i). This implies $\myPsi_m(\theta) = \myPsi_m(\theta_{a,0})$
  because $\Deq_m \cap \Sym_m^1$ is a complete representative system of
  $\Deq_m/\sim$. From (v) and $\gcd(a,m)=1$, $\theta \in 
  \Rcr_m^{\mathrm{L}, 0} \sqcup \Rcr_m^{\mathrm{L}, 1}$
   holds.
\end{proof}

The above lemma is established in terms of $\Rcr_m$ and $\Deq_m$,
independent of the knowledge on $\Sinv_m$.
To complete the description of properties of $\Rcr_m$,
we will quote a few results on 
$\Sinv_m$ (Theorem \ref{Known2} and Corollary \ref{Cor1}).
\begin{lem}\label{lem:projplus}
  Let $m \geq 3.$
The following holds:
\begin{itemize}
\item[(i)]
  The injection
  $\myPsi_{m}|_{\Rcr_m \setminus \Rcr_m^{\mathrm{L}, b}}$
  is a bijection from
  $\Rcr_m \setminus \Rcr_m^{\mathrm{L}, b}$
      to $\Sym_m^1 \cap \Deq_m$ for each of $b = 0, 1$.
      The common cardinality of the domain and the image is
      $\sum_{k=1}^{m-1} \phi(k).$
\item[(ii)]
    For every $\theta \in
      \Rcr_m
      \setminus (  \Rcr_m^{\mathrm{L}, 0} \sqcup \Rcr_m^{\mathrm{L}, 1} ),$
  the congruential difference set
  $\CDS_m(\myPsi_m(\theta))$
  is of the form
  $\{a, a+1\}$ for some $a \in [m-1]$.
   Moreover,
   $\{\theta' \in \Deq_m \cap \Sym_m^1\ : \ \exists\ a \in [m-1]
   \text{\ s.t. } \CDS_m(\theta')
   = \{a, a+1\} \} $
    is the bijective image of
   $\Rcr_m
    \setminus (  \Rcr_m^{\mathrm{L}, 0} \sqcup \Rcr_m^{\mathrm{L}, 1})$
    by $\myPsi_m$.
\end{itemize}
\end{lem}
\begin{proof}
(i)
  By Theorem \ref{Known2} and Corollary \ref{Cor1},
  $|\Rcr_{m}| = |\Sinv_{m}| = \sum_{k=1}^{m} \phi(k)$. 
  Thus we have, for each of $b=0, 1$, 
  $|\Rcr_m \setminus \Rcr_m^{\mathrm{L}, b}| =
    \sum_{k=1}^{m} \phi(k) - \phi(m) = \sum_{k=1}^{m-1} \phi(k).$
    Since the map
    $\myPsi_{m}|_{\Rcr_m \setminus \Rcr_m^{\mathrm{L}, b}}$
      is injective as seen in Lemma \ref{lem:proj} (iii),
      we have $|\Deq_m \cap \Sym_m^1| \geq \sum_{k=1}^{m-1} \phi(k)$.
      The reversed
      inequality (multiplied by $m$) has already been 
      shown in the very end of the proof of
      Theorem \ref{teiri2}, yielding the equality
      $|\Deq_m \cap \Sym_{m}^{1}|=\sum_{k=1}^{m-1} \phi(k)$.
    Therefore, $\myPsi_{m}|_{\Rcr_m \setminus \Rcr_m^{\mathrm{L}, b}}$
      is bijective for $b=0,1$.

(ii)
  Let
  $A' :=\{ \theta' \in \Deq_m \cap \Sym_m^1 \ : 
\  \exists\ a \in [m-1] \text{\ s.t. } \CDS_m(\theta')
   = \{a\} \} $
  and
  $A'' :=\{ \theta' \in \Deq_m \cap \Sym_m^1 \ : 
\  \exists\ a \in [m-1] \text{\ s.t. } \CDS_m(\theta')
   = \{a, a+1\} \} $
   respectively.
  By the definition of $\Rcr_m$,
namely the defining congruential recurrence \eqref{knowneq1},
   the congruential difference set 
   $\CDS_m(\theta)$
   is of the form either
   $\{a, a+1\}$ or $\{a\}$ for some $a \in [m-1],$ for all
   $\theta \in \Rcr_m$. Because $\myPsi_m(\theta)$ is a shift of $\theta$,
   it follows that
   $\CDS_m(\myPsi_m(\theta)) = \CDS_m(\theta).$
   On the other hand, the assertion (i) just proved
   and Lemma \ref{lem:proj} (ii) tell
   us that $\myPsi_m(\Rcr_m) = \Deq_m \cap \Sym_m^1$
   in particular.
   Hence, we have
   \begin{equation}\label{zounosabunhadotirakaninarutousiki}
   \Deq_m \cap \Sym_m^1 =\myPsi_m(\Rcr_m) = A' \sqcup A''.
   \end{equation}
   By Lemma \ref{lem:proj} (vi), for  $\theta \in \Rcr_m$,
   $\CDS_m(\myPsi_m(\theta)) $ is a singleton $\{a\}$ if and only if
   $\theta \in \Rcr_m^{\mathrm{L}, 0} \sqcup \Rcr_m^{\mathrm{L}, 1}$.
   In other words, it follows that
   \begin{equation}\label{zounosabunhasingurutontousiki}
   \myPsi_m(\Rcr_m^{\mathrm{L}, 0} \sqcup \Rcr_m^{\mathrm{L}, 1})= A'.
   \end{equation}
   From Eqs.~\eqref{zounosabunhadotirakaninarutousiki},
   \eqref{zounosabunhasingurutontousiki} and
   Lemma \ref{lem:proj} (v), it follows that
   $\myPsi_m(\Rcr_m \setminus (\Rcr_m^{\mathrm{L}, 0} \sqcup \Rcr_m^{\mathrm{L}, 1})) = A'',$
   i.e.,
$\myPsi_m|_{\Rcr_m \setminus (\Rcr_m^{\mathrm{L}, 0} \sqcup \Rcr_m^{\mathrm{L}, 1})}: \Rcr_m \setminus (\Rcr_m^{\mathrm{L}, 0} \sqcup \Rcr_m^{\mathrm{L}, 1}) \rightarrow A''$
is a surjection.
   For the bijectivity,
  $\myPsi_m|_{\Rcr_m \setminus (\Rcr_m^{\mathrm{L}, 0} \sqcup \Rcr_m^{\mathrm{L}, 1})}$ is injective, again by the assertion (i)
   and
   $\Rcr_m \setminus (\Rcr_m^{\mathrm{L}, 0} \sqcup \Rcr_m^{\mathrm{L}, 1})
 \subset \Rcr_m \setminus \Rcr_m^{\mathrm{L}, 1}
   $.
\end{proof}
\begin{rem}
  The validity of Lemmas \ref{lem:proj} and \ref{lem:projplus}
  is able to be verified for the case $m=2$ where
  $\Rcr_m = \Deq_m = \Sym_2$, by a direct argument.
\end{rem}  
\begin{rem}
  The above proof of (i) depends on 
  Theorem \ref{Known2}\ and Corollary \ref{Cor1}.
  However, what we really needed in the proof
  was the fact
  $|\Rcr_m| \geq \sum_{k=1}^{m} \phi(k)$ only. We used 
  no detailed property of $\Sinv_m, \Sos_m$ nor Farey sequence associated
  to them.
\end{rem}

\subsection{The lifting procedure}
Now we have the building block $\myPsi_{m}|_{\Rcr_m}$
which maps $\Rcr_m$ of cardinality $\sum_{k=1}^{m} \phi(k)$
onto $\Sym_m^1 \cap \Deq_m$ of cardinality $\sum_{k=1}^{m-1} \phi(k)$.
It is not injective but the collisions are well-controlled, as mentioned in
Lemma \ref{lem:proj}.

On the other hand, we already have had another ingredient, the bijection
$\myPhi_m|_{\Sym_m^1 \cap \Deq_m}: \Sym_m^1 \cap \Deq_m \rightarrow
\Rcr_{m-1}.$ Then, taking the composition,
we obtain the surjection
$\myPhi_m|_{\Sym_m^1 \cap \Deq_m} \circ \myPsi_{m}|_{\Rcr_m}
: \Rcr_m \rightarrow \Rcr_{m-1}$ in which the only
source of collisions is the $2:1$ property of 
$\myPsi_{m}|_{\Rcr_m^{\mathrm{L},0}  \sqcup \Rcr_m^{\mathrm{L},1}}$.

The composition projects $\Rcr_m$ to $\Rcr_{m-1}$ nicely.
Indeed,
up to the $2:1$ case of Lemma \ref{lem:proj} 
(iv), the projection is invertible.
Thus we are led to the following procedure to lift
the permutations $\pi$ in $\Rcr_{m-1}$ to form
$\Rcr_{m}$.

\begin{proc}\label{proc:genproc}
  Given $m \geq 2,$ do the following:
\begin{itemize}
\item[\tt 1]  Let $X := \emptyset;$
\item[\tt 2] For each $\pi \in \Rcr_{m-1}$ do $\{$
\item[\tt 3] $\quad$ Let $\theta_\pi := \myPhi_m^{-1}(\pi);$
\item[\tt 4] $\quad$ If $
\CDS_m(\theta_\pi)
  = \{a\}$ then
\item[\tt 5] $\quad\quad \{$ \ Let $Y := \{\theta_{a,0}, \theta_{a,1}\}; \ \}$
\item[\tt 6] $\quad$ If $
\CDS_m(\theta_\pi)
  = \{a, a+1\}$ then
\item[\tt 7]  $\quad\quad \{$ \ Let
  $\hat{\theta} := \supermodtwo{m}{\theta_\pi(\cdot) + a}$ and
  let $Y := \{\hat{\theta}\}; \  \}$
\item[\tt 8] $\quad$ Let $X := X \cup Y$.
\item[\tt 9] $\}$
\item[\tt 10] Let $\Rcr_{m} := X$.
\end{itemize}
Here we put
    ${\Rcr}_{1} := \Sym_1, {\Req}_{1} := \Sym_1$ and ${\Deq}_{2} := \Sym_2$,
    for which Proposition \ref{Prop1} holds,
    when $m=2$.
\end{proc}

\begin{thm}
  Procedure \ref{proc:genproc} lifts $\Rcr_{m-1}$ to
  $\Rcr_{m}$ correctly.
\end{thm}
\begin{proof}
  In the line 3, $\theta_\pi \in \Sym_{m}^1 \cap \Deq_m$
is uniquely determined
    since 
    $\myPhi_m: \Sym_m^1 \rightarrow
    \Sym_{m-1}$ is bijective and so is
    $\myPhi_m|_{\Sym_m^1 \cap \Deq_m}: \Sym_m^1 \cap \Deq_m \rightarrow
      \Rcr_{m-1}$ by Proposition \ref{Prop1} and Theorem \ref{teiri1}.
    From 
Lemma \ref{lem:proj} (vi) and Lemma \ref{lem:projplus} (ii),
exactly one ``If'' statement of the lines 4 and 6 is the case.
When the statement in the line 4 is the case,
$\theta_\pi = \theta_{a, b}$ for some $b \in [m]$, however,
$b = m-a+1$ from the condition $\theta_\pi(1) = 1.$ Thus, from
Lemma \ref{lem:proj} (iv) 
and (vi),
$Y = (\myPsi_m|_{\Rcr_m})^{-1}(\{ \theta_\pi \})$ holds in the line 5.
On the other hand, when the statement in the line 6 holds,
by Lemma \ref{lem:projplus} (ii),
a unique $\theta' := \myPsi_m^{-1}(\theta_\pi) \in 
\Rcr_m \setminus (\Rcr_m^{\mathrm{L}, 0} \sqcup \Rcr_m^{\mathrm{L}, 1})$
exists.
From $\theta' \in \Rcr_m$ and the shift-equivalence
$\theta'(i+1) - \theta'(i) \equiv \theta_\pi(i+1) - \theta_\pi(i) \mod m \ (i \in [m-1])$,
it follows that
$\{ \supermodtwo{m}{\theta'(1) - \mychi(\theta'(m) \le \theta'(i))} : i \in [m-1]\} = \{ \supermodtwo{m}{\theta'(i+1) - \theta'(i)} : i \in [m-1]\}
= \{ \supermodtwo{m}{\theta_\pi(i+1) - \theta_\pi(i)} : i \in [m-1]\}
= \{a, a+1\}$. Thus, taking the maximum of the last set of cardinality 2,
$\theta'(1) = a + 1$ which is $a + \theta_\pi(1)$ by $\theta_\pi \in \Sym_m^1$.
Therefore, $\theta' = \supermodtwo{m}{\theta_\pi(\cdot) + a},$
which is nothing but $\hat{\theta}$ in the line 7.

Thus the lines 4,5,6 and 7 together compute the correct inverse image
$Y=(\myPsi_m|_{\Rcr_m})^{-1}(\{\theta_\pi\})$ for $\theta_\pi=\myPhi_m^{-1}(\pi)$.
Then the set $X$ in the line 10 is
$\bigcup_{\pi \in \Rcr_{m-1}} (\myPsi_m|_{\Rcr_m})^{-1}(\{\myPhi_m^{-1}(\pi)\})
= (\myPhi_m|_{\Deq_m \cap \Sym_m^1}\circ \myPsi_m|_{\Rcr_m})^{-1}(\Rcr_{m-1})
= \Rcr_m$.
\end{proof}  

\subsection{Recursion of the lifting}\label{ss:recursionofthelifting}
\begin{figure}[t]
\centering
  \rotatebox{0}{\scalebox{0.65}{
   \begin{tikzpicture}
     \tikzset{level distance=40pt}
\Tree [ .\node(V1){$\begin{matrix}\mathbf{1}\\\end{matrix}$};
 [ .\node(Y2){$\begin{matrix}\mathbf{12}\\\end{matrix}$};
  [ .\node(V2){$\begin{matrix}\mathbf{12}^{(0)}\\\end{matrix}$};
   [ .\node(Y3){$\begin{matrix}\mathbf{123}\\\end{matrix}$};
    [ .\node(V3){$\begin{matrix}\mathbf{123}^{(0)}\\\end{matrix}$};
     [ .\node(Y4){$\begin{matrix}\mathbf{1234}\\\end{matrix}$};
      [ .\node(V4){$\begin{matrix}\mathbf{1234}^{(0)}\\\end{matrix}$};
       [ .\node(Y5){$\begin{matrix}\mathbf{12345}\\\end{matrix}$};
        [ .\node(V5){$\begin{matrix}\mathbf{12345}^{(0)}\\\end{matrix}$};
         [ .\node(Y6){$\begin{matrix}\mathbf{123456}\\\end{matrix}$};
          [ .\node(V6){$\begin{matrix}\mathbf{123456}^{(0)}\\\end{matrix}$}; \edge[dashed] node [auto=left] {};{$\cdot$} ]
          [ .\node[]{$\begin{matrix}\mathbf{234561}^{(1)}\\\end{matrix}$}; \edge[dashed] node [auto=left] {};{$\cdot$} ]
         ]
        ]
        [ .\node[]{$\begin{matrix}\mathbf{23451}^{(1)}\\\end{matrix}$};
         [ .\node[]{$\begin{matrix}\mathbf{134562}\\\end{matrix}$};
          [ .\node[]{$\begin{matrix}\mathbf{245613}\\\end{matrix}$}; \edge[dashed] node [auto=left] {};{$\cdot$} ]
         ]
        ]
       ]
      ]
      [ .\node[]{$\begin{matrix}\mathbf{2341}^{(1)}\\\end{matrix}$};
       [ .\node[]{$\begin{matrix}\mathbf{13452}\\\end{matrix}$};
        [ .\node[]{$\begin{matrix}\mathbf{24513}\\\end{matrix}$};
         [ .\node[]{$\begin{matrix}\mathbf{135624}\\\end{matrix}$};
          [ .\node[]{$\begin{matrix}\mathbf{246135}\\\end{matrix}$}; \edge[dashed] node [auto=left] {};{$\cdot$} ]
         ]
        ]
       ]
      ]
     ]
    ]
    [ .\node[]{$\begin{matrix}\mathbf{231}^{(1)}\\\end{matrix}$};
     [ .\node[]{$\begin{matrix}\mathbf{1342}\\\end{matrix}$};
      [ .\node[]{$\begin{matrix}\mathbf{2413}\\\end{matrix}$};
       [ .\node[]{$\begin{matrix}\mathbf{13524}\\\end{matrix}$};
        [ .\node[]{$\begin{matrix}\mathbf{24135}^{(0)}\\\end{matrix}$};
         [ .\node[]{$\begin{matrix}\mathbf{135246}\\\end{matrix}$};
          [ .\node[]{$\begin{matrix}\mathbf{351462}\\\end{matrix}$}; \edge[dashed] node [auto=left] {};{$\cdot$} ]
         ]
        ]
        [ .\node[]{$\begin{matrix}\mathbf{35241}^{(1)}\\\end{matrix}$};
         [ .\node[]{$\begin{matrix}\mathbf{146352}\\\end{matrix}$};
          [ .\node[]{$\begin{matrix}\mathbf{362514}\\\end{matrix}$}; \edge[dashed] node [auto=left] {};{$\cdot$} ]
         ]
        ]
       ]
      ]
     ]
    ]
   ]
  ]
  [ .\node[]{$\begin{matrix}\mathbf{21}^{(1)}\\\end{matrix}$};
   [ .\node[]{$\begin{matrix}\mathbf{132}\\\end{matrix}$};
    [ .\node[]{$\begin{matrix}\mathbf{213}^{(0)}\\\end{matrix}$};
     [ .\node[]{$\begin{matrix}\mathbf{1324}\\\end{matrix}$};
      [ .\node[]{$\begin{matrix}\mathbf{3142}\\\end{matrix}$};
       [ .\node[]{$\begin{matrix}\mathbf{14253}\\\end{matrix}$};
        [ .\node[]{$\begin{matrix}\mathbf{31425}^{(0)}\\\end{matrix}$};
         [ .\node[]{$\begin{matrix}\mathbf{142536}\\\end{matrix}$};
          [ .\node[]{$\begin{matrix}\mathbf{415263}\\\end{matrix}$}; \edge[dashed] node [auto=left] {};{$\cdot$} ]
         ]
        ]
        [ .\node[]{$\begin{matrix}\mathbf{42531}^{(1)}\\\end{matrix}$};
         [ .\node[]{$\begin{matrix}\mathbf{153642}\\\end{matrix}$};
          [ .\node[]{$\begin{matrix}\mathbf{426315}\\\end{matrix}$}; \edge[dashed] node [auto=left] {};{$\cdot$} ]
         ]
        ]
       ]
      ]
     ]
    ]
    [ .\node[]{$\begin{matrix}\mathbf{321}^{(1)}\\\end{matrix}$};
     [ .\node[]{$\begin{matrix}\mathbf{1432}\\\end{matrix}$};
      [ .\node[]{$\begin{matrix}\mathbf{3214}^{(0)}\\\end{matrix}$};
       [ .\node[]{$\begin{matrix}\mathbf{14325}\\\end{matrix}$};
        [ .\node[]{$\begin{matrix}\mathbf{42153}\\\end{matrix}$};
         [ .\node[]{$\begin{matrix}\mathbf{153264}\\\end{matrix}$};
          [ .\node[]{$\begin{matrix}\mathbf{531642}\\\end{matrix}$}; \edge[dashed] node [auto=left] {};{$\cdot$} ]
         ]
        ]
       ]
      ]
      [ .\node[]{$\begin{matrix}\mathbf{4321}^{(1)}\\\end{matrix}$};
       [ .\node[]{$\begin{matrix}\mathbf{15432}\\\end{matrix}$};
        [ .\node[]{$\begin{matrix}\mathbf{43215}^{(0)}\\\end{matrix}$};
         [ .\node[]{$\begin{matrix}\mathbf{154326}\\\end{matrix}$};
          [ .\node[]{$\begin{matrix}\mathbf{532164}\\\end{matrix}$}; \edge[dashed] node [auto=left] {};{$\cdot$} ]
         ]
        ]
        [ .\node[]{$\begin{matrix}\mathbf{54321}^{(1)}\\\end{matrix}$};
         [ .\node[]{$\begin{matrix}\mathbf{165432}\\\end{matrix}$};
          [ .\node[]{$\begin{matrix}\mathbf{543216}^{(0)}\\\end{matrix}$}; \edge[dashed] node [auto=left] {};{$\cdot$} ]
          [ .\node[]{$\begin{matrix}\mathbf{654321}^{(1)}\\\end{matrix}$}; \edge[dashed] node [auto=left] {};{$\cdot$} ]
         ]
        ]
       ]
      ]
     ]
    ]
   ]
  ]
 ]
]
     \node (orig) at (-12,-12) {};
     \node (V10) at (V1 -| orig) {$\Rcr_1$};
     \draw [dotted, gray] (V1) -- (V10);
     \foreach \m in {2,3,4,5,6}
     {
     \node (Y\m0) at (Y\m -| orig) {$\Deq_\m \cap \Sym_\m^1$};
     \node (V\m0) at (V\m -| orig) {$\Rcr_\m$};
     \draw [dotted, gray] (Y\m) -- (Y\m0);
     \draw [dotted, gray] (V\m) -- (V\m0);
     }
     \foreach \n/\m in {1/2,2/3,3/4,4/5,5/6}
     {
     \draw [->] (V\m0) -- (Y\m0) ;
     \node at ($(V\m0)!0.5!(Y\m0)$) [left] {$\myPsi_\m$};
     \draw [->] (Y\m0) -- (V\n0) ;
     \node at ($(Y\m0)!0.5!(V\n0)$) [left] {$\myPhi_\m$};
     }
     \node (DDD) at (V60)  [below] {$\vdots$};
   \end{tikzpicture}
}}
  \caption{Recursive application of $\myPhi\circ\myPsi$ or its inverse}
\label{fig:recursiveone}  
\end{figure}
In this subsection, we assume that $m \geq 2$.
The composition
$\Rcr_{m} \overset{\myPsi_m}{\rightarrow}
\Deq_m \cap \Sym_m^1
\overset{\myPhi_m}{\rightarrow} \Rcr_{m-1}$ gives rise to
the projection $\Rcr_{m} \rightarrow \Rcr_{m-1}$ and
it can be applied recursively to reach to
$\Rcr_{1}=\Sym_{1}$ containing only one element $\mathrm{id}|_{[1]} \in \Sym_1$.
In other words, starting with $\Rcr_{1}$, inverting the arrows
by the iterative application of the lifting, Procedure \ref{proc:genproc},
we obtain a generation process of
$\Rcr_{1} \rightarrow \Rcr_{2} \rightarrow \cdots \rightarrow \Rcr_{m}$
for arbitrary degree $m$.

Figure \ref{fig:recursiveone} shows the tree corresponding to the process
for $m \leq 6$.
It has two interpretations; ascending the family tree (projection)
and descending the family tree (generation). Here we explain it
based on the latter interpretation.

Permutations are shown using the one-line notation and a bold font.
At the root of the tree,
the unique element $\mathbf{1} \in \Rcr_1$
exists. 
For $m=2,$ the line 3 of Procedure \ref{proc:genproc}
is applied to $\pi = \mathbf{1} \in \Rcr_1$, yielding
$\theta_\pi = \mathbf{12} \in \Deq_2 \cap \Sym_2^1$.
For this $\theta_\pi$, ``If'' statement in the line 4
is the case with $a = 1.$ Thus, the line 5 produces two elements
$\theta_{1,0} = \mathbf{12}$ and
$\theta_{1,1} = \mathbf{21} \in \Rcr_2$. In the picture,
they are denoted as  $\mathbf{12}^{(0)}$ and $\mathbf{21}^{(1)}$
respectively, where the index $(b)$ indicates the second parameter $b$ of
$\theta_{a,b}$. When we arrange these nodes in the tree, there are two ways,
however we employ the ordering $(0)$-first (left), $(1)$-second (right).
Since 
$\theta_\pi = \mathbf{12}$  is the only
element of $\Deq_2 \cap \Sym_2^1$, we reach to
the line 10 with $X=\{\mathbf{12}, \mathbf{21}\}$.
Thus, we have constructed $\Rcr_2 = \{\mathbf{12}, \mathbf{21}\}$.

The run of Procedure \ref{proc:genproc} for $m=3$ is similar
and we have $\Rcr_3 = \{\mathbf{123}, \mathbf{231}, \mathbf{213}, \mathbf{321} \}$ at the line 10.

For $m=4$, however, 
for some $\theta_\pi \in \Deq_4 \cap \Sym_4^1$,
namely for $\theta_\pi = \mathbf{1342}, \mathbf{1324}
$ in Fig. \ref{fig:recursiveone},
the line 6 is the case and for each of them, a single child
$\hat{\theta} \in \Rcr_4$
is produced($\mathbf{2413}$ for $\theta_\pi = \mathbf{1342}$ and
$\mathbf{3142}$
for $\theta_\pi = \mathbf{1324}$).

Iterating the lifting procedure, we finally have
$\sum_{k=1}^{6} \phi(k) = 1 + 1 + 2 + 2 + 4 + 2= 12$ elements
of $\Rcr_6$.

\begin{rem}\label{rem:genproc-farey}
  By using Procedure \ref{proc:genproc} and Corollary \ref{Cor1},
  we are able to construct $\Sinv_m$, the inverses of the
   {\mysos} permutations for arbitrary degree $m$,
         {\it without depending on\/}  Farey sequence.
  If we may depend on Farey sequence, such construction is straightforward:
  By {\mysur}'s bijection \cite[Sats I]{sur},
  the denominators of successive two fractions in \Th{$m$} Farey
  sequence correspond to the 1st and \Th{$m$} terms of a {\mysos}
  permutation respectively, while a {\mysos} permutation is 
  uniquely determined by its 1st and \Th{$m$} terms, by the recurrence
  of Theorem \ref{Known0}. Using this bijection, the entire $\Sos_m$,
  the set of {\mysos} permutations of degree $m$, is constructable
  from \Th{$m$} Farey sequence.
    We also note that 
    even the recurrence \eqref{knowneq1}
which defines the set ${\Rcr}_m$,
does not appear in Procedure \ref{proc:genproc}.
\end{rem}
\newcommand \fperm[2]{\mathsf{F}_{#1}[#2]}
\newcommand \permf[2]{\mathsf{F}_{#1}^{-1}[#2]}
\newcommand \downa[1]{\hspace{-0.5ex}\downarrow\hspace{-0.5ex}{}^{#1}}
\newsavebox{\Bmychi}
\savebox{\Bmychi}{$\mychi$}
\renewcommand{\mychi}{\usebox{\Bmychi}}
%
%
\subsection{Generation Tree vs. Farey Tree}\label{ss:orders}

Finally, we mention an isomorphism between
  the generation tree of Procedure \ref{proc:genproc} and
  a
  tree formed by the Farey intervals.
  
  For an integer $M \geq 1$, let $T_M$ be the binary tree
  of Procedure \ref{proc:genproc}
  where the nodes of $\Rcr_{m}$ for $m \leq M$ are present
  but all the nodes of $\Deq_{m} \cap \Sym_{m}^{1}$ are omitted.
  Formally, $T_M$ has $M$ levels $m \in  [M],$
  its root node (level 1) of $T_M$ is
  $\mathbf{1} \in \Rcr_{1} = \Sym_1$ and its level $m$ nodes
  are the elements of $\Rcr_{m}$, where, for $m \geq 2$, the nodes
  $\pi \in \Rcr_{m-1}$ and $\hat{\theta} \in \Rcr_{m}$
  are connected by an edge
  if and only if $\myPhi_{m}\circ \myPsi_{m}(\hat{\theta}) = \pi$.
  When $\pi \in \Rcr_{m-1}$ has two children
  $\{\theta_{a,0}, \theta_{a,1}\} =
   \myPsi_{m}^{-1}\circ\myPhi_{m}^{-1}(\{\pi\})$, $\pi$ is said to be
 \textit{branching\/}. In this case,
  the two descendants of $\pi$
 are drawn $\theta_{a,0}$-left, $\theta_{a,1}$-right, as in Fig.~\ref{fig:recursiveone}.
 As an example, Figure~\ref{fig:recursivetwo} shows $T_6$
 which corresponds to Fig.~\ref{fig:recursiveone}.
 \begin{figure}[t]
\centering
  \rotatebox{0}{\scalebox{0.65}{
   \begin{tikzpicture}
     \tikzset{level distance=40pt}
%
%
\Tree [ .\node(V1){$\begin{matrix}\mathbf{1}\\\end{matrix}$};
  [ .\node(V2){$\begin{matrix}\mathbf{12}^{(0)}\\\end{matrix}$};
    [ .\node(V3){$\begin{matrix}\mathbf{123}^{(0)}\\\end{matrix}$};
      [ .\node(V4){$\begin{matrix}\mathbf{1234}^{(0)}\\\end{matrix}$};
        [ .\node(V5){$\begin{matrix}\mathbf{12345}^{(0)}\\\end{matrix}$};
          [ .\node(V6){$\begin{matrix}\mathbf{123456}^{(0)}\\\end{matrix}$}; \edge[dashed] node [auto=left] {};{$\cdot$} ]
          [ .\node[]{$\begin{matrix}\mathbf{234561}^{(1)}\\\end{matrix}$}; \edge[dashed] node [auto=left] {};{$\cdot$} ]
        ]
        [ .\node[]{$\begin{matrix}\mathbf{23451}^{(1)}\\\end{matrix}$};
          [ .\node[]{$\begin{matrix}\mathbf{245613}\\\end{matrix}$}; \edge[dashed] node [auto=left] {};{$\cdot$} ]
        ]
      ]
      [ .\node[]{$\begin{matrix}\mathbf{2341}^{(1)}\\\end{matrix}$};
        [ .\node[]{$\begin{matrix}\mathbf{24513}\\\end{matrix}$};
          [ .\node[]{$\begin{matrix}\mathbf{246135}\\\end{matrix}$}; \edge[dashed] node [auto=left] {};{$\cdot$} ]
        ]
      ]
    ]
    [ .\node[]{$\begin{matrix}\mathbf{231}^{(1)}\\\end{matrix}$};
      [ .\node[]{$\begin{matrix}\mathbf{2413}\\\end{matrix}$};
        [ .\node[]{$\begin{matrix}\mathbf{24135}^{(0)}\\\end{matrix}$};
          [ .\node[]{$\begin{matrix}\mathbf{351462}\\\end{matrix}$}; \edge[dashed] node [auto=left] {};{$\cdot$} ]
        ]
        [ .\node[]{$\begin{matrix}\mathbf{35241}^{(1)}\\\end{matrix}$};
          [ .\node[]{$\begin{matrix}\mathbf{362514}\\\end{matrix}$}; \edge[dashed] node [auto=left] {};{$\cdot$} ]
        ]
      ]
    ]
  ]
  [ .\node[]{$\begin{matrix}\mathbf{21}^{(1)}\\\end{matrix}$};
    [ .\node[]{$\begin{matrix}\mathbf{213}^{(0)}\\\end{matrix}$};
      [ .\node[]{$\begin{matrix}\mathbf{3142}\\\end{matrix}$};
        [ .\node[]{$\begin{matrix}\mathbf{31425}^{(0)}\\\end{matrix}$};
          [ .\node[]{$\begin{matrix}\mathbf{415263}\\\end{matrix}$}; \edge[dashed] node [auto=left] {};{$\cdot$} ]
        ]
        [ .\node[]{$\begin{matrix}\mathbf{42531}^{(1)}\\\end{matrix}$};
          [ .\node[]{$\begin{matrix}\mathbf{426315}\\\end{matrix}$}; \edge[dashed] node [auto=left] {};{$\cdot$} ]
        ]
      ]
    ]
    [ .\node[]{$\begin{matrix}\mathbf{321}^{(1)}\\\end{matrix}$};
      [ .\node[]{$\begin{matrix}\mathbf{3214}^{(0)}\\\end{matrix}$};
        [ .\node[]{$\begin{matrix}\mathbf{42153}\\\end{matrix}$};
          [ .\node[]{$\begin{matrix}\mathbf{531642}\\\end{matrix}$}; \edge[dashed] node [auto=left] {};{$\cdot$} ]
        ]
      ]
      [ .\node[]{$\begin{matrix}\mathbf{4321}^{(1)}\\\end{matrix}$};
        [ .\node[]{$\begin{matrix}\mathbf{43215}^{(0)}\\\end{matrix}$};
          [ .\node[]{$\begin{matrix}\mathbf{532164}\\\end{matrix}$}; \edge[dashed] node [auto=left] {};{$\cdot$} ]
        ]
        [ .\node[]{$\begin{matrix}\mathbf{54321}^{(1)}\\\end{matrix}$};
          [ .\node[]{$\begin{matrix}\mathbf{543216}^{(0)}\\\end{matrix}$}; \edge[dashed] node [auto=left] {};{$\cdot$} ]
          [ .\node[]{$\begin{matrix}\mathbf{654321}^{(1)}\\\end{matrix}$}; \edge[dashed] node [auto=left] {};{$\cdot$} ]
        ]
      ]
    ]
  ]
]
     \node (orig) at (-12,-8) {};
     \node (V10) at (V1 -| orig) {$\Rcr_1$};
     \draw [dotted, gray] (V1) -- (V10);
     \foreach \m in {2,3,4,5,6}
     {
     \node (V\m0) at (V\m -| orig) {$\Rcr_\m$};
     \draw [dotted, gray] (V\m) -- (V\m0);
     }
     \foreach \n/\m in {1/2,2/3,3/4,4/5,5/6}
     {
     \draw [->] (V\m0) -- (V\n0) ;
     \node at ($(V\m0)!0.5!(V\n0)$) [left] {$\myPhi_\m \circ \myPsi_\m$};
     }
     \node (DDD) at (V60)  [below] {$\vdots$};
   \end{tikzpicture}
}}
  \caption{The tree $T_6$}
\label{fig:recursivetwo}  
\end{figure}
 
Though the description of Theorem \ref{Known2}
 stating the exact enumeration of {\mysos} permutations
 was sufficient
to derive the results in Section \ref{sect:mainresults},
in order to explain the tree formed by the Farey intervals,
we need a more explicit description of how \mysur{}'s bijection
mentioned at the beginning of Section \ref{sect:known}
implies Theorem \ref{Known2}. The refined statement
of Theorem \ref{Known2}
is as follows:
\begin{known}[{\cite[Satz I]{sur}}, see also \cite{rosia}]\label{known:sur}
  Let $m \geq 2$ and let $N := \sum_{l=1}^{m} \phi(l).$ Let
  $0 = f_0 < f_1 < \cdots
  < f_{N} = 1$ be the
\Th{$m$}~Farey sequence and
  let $\mathfrak{F}_m = \{ F_t = (f_{t-1}, f_{t}) : t \in \left[ N \right] \}$
  be the collection of the open intervals formed by 
  two successive terms of them.
  Let $\sigma_\alpha \in \Sos_m$ be the {\mysos} permutation
  satisfying Ineq.~\eqref{Sosnojyouken1}
  where $\alpha \in F$ for some $F \in \mathfrak{F}_m$ and
  let $\tau_\alpha^m = \sigma_\alpha^{-1}$ be its inverse.
  
  When $\alpha \in F$ and  $\alpha' \in F'$
  hold for $\alpha, \alpha' \in (0,1)$
  and
  $F, F' \in \mathfrak{F}_m$, the two permutations
  $\tau_\alpha^m$ and
  $\tau_{\alpha'}^m$ are equal if and only if $F = F'$.
  In other words, the map
  $\tau_\cdot^m:\cup_{F \in \mathfrak{F}_m}  F \ni \alpha \mapsto
  \tau_\alpha^m \in \Sinv_m$
  induces a bijection
  $\mathsf{F}_{m}: \Sinv_m \rightarrow \mathfrak{F}_{m}.$
  In particular, $|\Sinv_m| = N.$
\end{known}  

By $\Rcr_m = \Sinv_m$ of Corollary \ref{Cor1}, this one-to-one correspondence
gives the bijection between $\Rcr_m$ and $\mathfrak{F}_m$:
\begin{equation}
\mathsf{F}_{m}: \Rcr_m \ni \pi \mapsto \fperm{m}{\theta} \in \mathfrak{F}_m.
\label{eq:surcorresp}
\end{equation}
Over the set $\mathfrak{F}_m$ indexed as in
    Theorem \ref{known:sur}, an obvious order 
\begin{equation}  
  F_t \preceq_{\mathfrak{F}_m} F_{t'}
  \overset{\mathrm{def}}{\Leftrightarrow} t \leq t'
\end{equation}
for Farey intervals ${F_t}=(f_{t-1}, f_{t})$ and ${F_{t'}}=(f_{t'-1}, f_{t'})$
is defined.

Now we are ready to explain the tree formed by the Farey intervals.
Let $T_{F,M}$ be the tree having $M$ levels, where
 the nodes of level $m \leq M$ are the Farey intervals in
 $\mathfrak{F}_{m}$. When $m \geq 2,$ nodes $F \in \mathfrak{F}_{m-1}$
 and $F' \in \mathfrak{F}_{m}$ are connected by an edge
 if and only if $F \supset F'$. The tree is also binary
 because $F \in \mathfrak{F}_{m-1}$
 contains at most two $F' \in \mathfrak{F}_{m}$.
 If $F\in  \mathfrak{F}_{m-1}$ contains two $F', F'' \in
 \mathfrak{F}_{m}$, then they are arranged in the $\preceq_{\mathfrak{F}_m}$
 order.

Also, let $T_{S,M}$ be the resulting tree
   obtained by replacing each node $F \in \mathfrak{F}_{m}$ for $m \in [M]$
   of 
   $T_{F,M}$ with 
  $\permf{m}{F} \in \Rcr_{m}$.

\begin{thm}\label{thm:eqvtrees}
  For each positive integer $M,$ the trees $T_M$ and $T_{S,M}$ are identical,
  including the horizontal order of nodes.
  In other words, $T_M$ and $T_{F,M}$
  are isomorphic as binary trees. In particular,
  in each level $m \in [M]$ of $T_M$,
  the horizontal order of the elements $\theta \in \Rcr_{m}$
  is the $\preceq_{\mathfrak{F}_{m}}$ order of
  the corresponding Farey intervals $\fperm{m}{\theta}$
  (of Eq.~\eqref{eq:surcorresp})
  in $\mathfrak{F}_{m}$.
\end{thm}    

The proof depends on the following propositions, whose
proofs are given in Appendix. In the sequel, the permutation
$\tau_\alpha^m \in \Sinv_m = \Rcr_m$ is as in Theorem \ref{known:sur}.
\begin{prp}\label{prop:compat}
  Let $m \geq 2$
and let $\alpha$ be a real number in $(0,1)$ which is not in the
\Th{$m$}~Farey sequence.
Then $\myPhi_m\circ\myPsi_m(\tau_\alpha^{m})
   = \tau_\alpha^{m-1}$.
\end{prp}
\begin{prp}[{A similar assertion was obtained in \cite[Theorem 5]{tikan1}}] \label{prop:eps}
  Suppose that $a \in [m], \gcd(a, m)=1$ and let $\alpha :=
   a/m$. Let $\theta_{a,b} \in \Rcr_m^{\mathrm{L}, b}, b = 0, 1$ be as in
   Proposition \ref{PropT}.
   Then, the following holds:
\begin{itemize}
\item[(i)] $\tau_\alpha = \theta_{a,1} \in \Rcr_m^{\mathrm{L},1}.$
  \textrm{(}Note: $\sigma_\alpha = \tau_\alpha^{-1}$ for $\alpha = a/m$
  slightly deviates from Ineq.~\eqref{Sosnojyouken1} in that
  $\{\sigma_\alpha(1)\alpha\} = \{ m a / m \} = 0.$ However, we can
  replace the leftmost $<$ in Ineq.~\eqref{Sosnojyouken1}
  with $\leq$
  without changing
  the set $\Sos_m$ hence $\Sinv_m.$ \textrm{)}
\item[(ii)] For any real number $\epsilon \in (0, 1/m^2)$,
  $\tau_{\alpha+\epsilon} = \theta_{a,1}$ whereas
  $\tau_{\alpha-\epsilon} = \theta_{a,0} \in \Rcr_m^{\mathrm{L},0}$.
\end{itemize}
\end{prp}
\begin{prp}\label{prop:division}
  Let $m \geq 2$ 
  and let $\pi$ be an element of  $\Rcr_{m-1}$.
  Let $\mathsf{F}_{m-1}:\Rcr_{m-1} \rightarrow \mathfrak{F}_{m-1}$
  be the bijection as in Eq.~\eqref{eq:surcorresp} and let
  $\theta_{a,b} \in \Rcr_m^{\mathrm{L}, b}, b = 0, 1$ be as in
   Proposition \ref{PropT}.  
  
  If $\pi$ is non-branching, then it holds 
  for the unique
    $\hat{\theta} \in \myPsi_{m}^{-1}\circ\myPhi_{m}^{-1}(\{ \pi
\})$ that 
  $\fperm{m}{\hat{\theta}} = \fperm{m-1}{\pi}.$ 

  If $\pi$ is branching, then there exists
  $a \in [m], \gcd(a,m) = 1$
  (thus $a/m$ is a new fraction in the
\Th{$m$}%
Farey sequence)
  such that $\myPsi_m^{-1} \circ \myPhi_m^{-1}(\{\pi\}) =
  \{\theta_{a,0}, \theta_{a,1} \}$
  and that
  $\fperm{m}{\theta_{a,0}} \sqcup \{a/m\}
    \sqcup
    \fperm{m}{\theta_{a,1}}
  = \fperm{m-1}{\pi}$.
\end{prp}

\subsubsection{Proof of Theorem \ref{thm:eqvtrees}}\label{ss:proofeqvtrees}

\begin{proof}
  When $M=1$, both $T_M$ and $T_{S,M}$ are the graph containing
  only one node $\mathbf{1} \in \Sym_1$ and no edge, thus they are equal.
  Suppose that $M\geq 2$ and that the assertion is valid for $M-1$. 
  Then, the nodes in level $M-1$ of the tree $T_{M-1}$
  are the leaves of $T_{M-1}$. Among them, 
look at the
\Th{$t$}%
node $\pi_t \in \Rcr_{M-1}$
  from the left.
By the induction hypothesis,
  $\pi_t$ is also the \Th{$t$}
node from the left among level $M-1$ nodes in $T_{S,M-1}$.

Now we compare the children of $\pi_t$ in $T_{M}$ and the children
of $\pi_t$ in $T_{S, M}$.
Note that $T_{M-1}$ (resp. $T_{S,M-1}$) is
the resulting subtree when level $M$ nodes
are removed from $T_{M}$ (resp. $T_{S,M}$) by definition.
Thus, $\pi_t$ is also
the \Th{$t$} node from the left among level $M-1$ nodes
both in $T_{M}$ and $T_{S,M}$.

  When $\pi_t$ is non-branching, $\pi_t$ has the unique child
  $\hat{\theta} \in \Rcr_{M}$ in the tree $T_{M}$.
  In the tree $T_{S,M}$, by Proposition \ref{prop:division},
  $\fperm{M-1}{\pi_t} = \fperm{M}{\hat{\theta}}$
  holds. Thus, $\fperm{M}{\hat{\theta}}$
  is the child of 
    $\fperm{M-1}{\pi_t}$ in $T_{F,M}$, implying that
  $\hat{\theta}$ is also the child of $\pi_t$ in $T_{S,M}$.
  When $\pi_t$ is branching, $\pi_t$ has the children
  $\theta_{a,0}$ (left) and $\theta_{a,1}$ (right) $\in \Rcr_{M}$
  (which are defined in Proposition \ref{PropT})
  in the tree $T_{M}$.
  By Proposition \ref{prop:division},
  $\fperm{M-1}{\pi_t} \supset
  \fperm{M}{\theta_{a,0}} \sqcup \fperm{M}{\theta_{a,1}}$
  holds (and there is no other $F' \in \mathfrak{F}_{M}$ such that
  $\fperm{M-1}{\pi_t} \supset F'$),
  thus 
  $\fperm{M}{\theta_{a,0}}$ and $\fperm{M}{\theta_{a,1}}$
  are the two children of $\fperm{M-1}{\pi_t}$ in  $T_{F,M}$.
  Therefore, $\theta_{a,0}$ and $\theta_{a,1}$ are
  the two children in the tree $T_{S,M}$,
  where $\theta_{a,0}$ is placed on the left side of
  $\theta_{a,1}$, 
  because the upper limit of $\fperm{M}{\theta_{a,0}}$
  and the lower limit of $\fperm{M}{\theta_{a,1}}$
  are both $a/m$ by Proposition \ref{prop:eps},
  showing
  that $\fperm{M}{\theta_{a,0}} \preceq_{\mathfrak{F}_M}
  \fperm{M}{\theta_{a,1}}$.

   Therefore, regardless of the number of the children of $\pi_t$
   which takes the same horizontal position both in $T_{M}$ and $T_{S,M}$,
   its descendants in $T_{M}$ and $T_{S,M}$ are the same
   including the horizontal order. Applying the argument to
   each of $\pi_t \in \Rcr_{M-1}$, 
   we have $T_{M} = T_{S,M}$.
\end{proof}
\section*{Acknowledgments}
The authors thank Shigeki Akiyama for his helpful suggestions
which let them consider the isomorphism problem of Section \ref{ss:orders}.

%
\appendix
\section{Proofs for the propositions used to prove Theorem \ref{thm:eqvtrees}}
In this appendix, we provide the proofs for
Propositions \ref{prop:compat}, \ref{prop:eps} and \ref{prop:division}
which were used to prove Theorem \ref{thm:eqvtrees} in
Section \ref{ss:proofeqvtrees}.

First we quote an explicit presentation
for a term of the inverse $\tau_\alpha^m  = \sigma_\alpha^{-1}$
of a \mysos{} permutation.
When there is no fear of confusion, $\tau_\alpha^m$ is denoted by
$\tau_\alpha.$
A direct consequence of Ineq.~\eqref{Sosnojyouken1} is
\begin{equation}
\tau_\alpha(i) = |\{j \in [m]: \{j \alpha\} \leq \{i \alpha\} \}| 
= \sum_{j = 1}^{m} \mychi(\{j \alpha\} \leq \{i \alpha\}),
\label{eq:taualphadef}
\end{equation}
however, in fact it has the following alternative presentation:
\begin{known}[{\cite[Propositions 3 and A1]{tikan1}}]
{%
Let $\alpha$ be a real number in $(0,1)$ which is not in the
\Th{$m$} Farey sequence.}
Let $\sigma_\alpha \in \Sos_m$ be the {\mysos} permutation
satisfying Ineq.~\eqref{Sosnojyouken1}
and let $\tau_\alpha := \sigma_\alpha^{-1}
  \in \Sinv_m.$ Then, it holds that
  \begin{equation}
    \tau_\alpha(i) = m (1 - \lfloor i \alpha \rfloor)
    + \sum_{j=1}^{m} \lfloor j \alpha \rfloor
    + \sum_{j=1}^{m} \lfloor (i-j) \alpha \rfloor
    \quad \text{for\ } i \in [m].
\label{eq:taualphaterms}
  \end{equation}
  In particular, it holds that
  \begin{equation}
    \tau_\alpha(1) = 1 + \lfloor m \alpha \rfloor
    \text{\ and\ } \tau_\alpha(m) = 2 m + 1 - (m + 1) \tau_\alpha(1)
    + 2 \sum_{j=1}^{m}
     \lfloor j \alpha \rfloor.
\label{eq:taualphaFirstLast}
  \end{equation}    
\end{known}

\subsection{Proof of Proposition \ref{prop:compat}}

\begin{proof}
  From Eq.~\eqref{eq:taualphaterms},
\begin{align}
\myPhi_m\circ\myPsi_m(\tau_\alpha^m)(i)
   & =
 \supermodtwo{m}{\tau_\alpha^m(i+1) - \tau_\alpha^m(1) + 1} - 1
\\
& \equiv m (1 - \lfloor (i+1) \alpha \rfloor) +
  \sum_{j=1}^{m} \lfloor j \alpha \rfloor
+ \sum_{j=1}^{m} \lfloor (i+1-j) \alpha \rfloor
    - 1 - \lfloor m \alpha \rfloor \mod m
    \\
 & \equiv 
     \sum_{j=1}^{m-1} \lfloor j \alpha \rfloor
   + \sum_{j=0}^{m-1} \lfloor (i-j) \alpha \rfloor
    - 1
    \mod m\\
    & \equiv
    \sum_{j=1}^{m-1} \lfloor j \alpha \rfloor
      + \ \lfloor i \alpha \rfloor
      + \sum_{j=1}^{m-1} \lfloor (i-j) \alpha \rfloor
      - 1
     \mod m
     \ \text{for\ } i \in [m-1]. 
\label{eq:compatone}
\end{align}
On the other hand, Eq.~\eqref{eq:taualphaterms}
 for $m - 1$ gives that
\begin{align}
\tau_\alpha^{m-1}(i) 
& =
(m-1) (1-\lfloor i \alpha\rfloor )
    + \sum_{j=1}^{m-1} \lfloor j \alpha \rfloor
    + \sum_{j=1}^{m-1} \lfloor (i-j) \alpha \rfloor \\
    & \equiv
-1+\lfloor i \alpha \rfloor
    + \sum_{j=1}^{m-1} \lfloor j \alpha \rfloor
    + \sum_{j=1}^{m-1} \lfloor (i-j) \alpha \rfloor 
         \mod m      
\ \text{for\ } i \in [m-1], 
\label{eq:compattwo}
\end{align}
which implies that
$\myPhi_m\circ\myPsi_m(\tau_\alpha^m)(i) \equiv \tau_\alpha^{m-1}(i) \mod m$
  for $i \in [m-1]$. The congruence is in fact equality since both LHS and RHS are in
 $[m-1] \subset \{0\} \cup [m-1]$.\end{proof}

\subsection{Proof of Proposition \ref{prop:eps}}
\begin{proof}
(i) 
  From
  Eq.~\eqref{eq:taualphadef},
\begin{equation}
  \tau_\alpha(i)  = \sum_{j=1}^{m}
  \mychi( \{ \alpha j \} \leq \{ \alpha i \}) 
   = \sum_{j=1}^{m}
  \mychi\left( \left\{ \frac{a j}{m} \right\} \leq \left\{ \frac{a i}{m}  \right\}\right) 
   = \sum_{j=1}^{m}
  \mychi\left(\frac{\mathrm{Mod}_m(a j)}{m} \leq \frac{\mathrm{Mod}_m(a i)}{m}\right).
\end{equation}
By $\gcd(a, m) = 1$, ${\mathrm{Mod}_m(a j)}/{m}$ for $j \in [m]$
have $m$ different values, therefore $\tau_\alpha$ is a permutation
and it holds that
\begin{align}
  \tau_\alpha(i)  & =
  \sum_{j=1}^{m} \mychi(\mathrm{Mod}_m(aj) \leq \mathrm{Mod}_m(ai)) \\
  & = \sum_{\ell=0}^{m-1}
  \mychi(\ell \leq \mathrm{Mod}_m(ai)) 
   = 1 + \mathrm{Mod}_m ( a i ) 
   = \supermodtwo{m}{ai + 1},
\end{align}
where the change of the variable $\ell:=\mathrm{Mod}_m(aj)$
is justified since 
$\mathrm{Mod}_m(aj)$ take all the $m$ values.

(ii) A simple but important fact is that 
$\left\lfloor ax/m \right\rfloor = \left\lfloor x (a \pm m \epsilon)/{m} \right\rfloor$
is the case
for 
$a, x \in \mathbb{Z} \setminus m \mathbb{Z}$ with $\gcd(a,m)=1,$ $|x| < m$
and $0 < \epsilon < 1/m^2$,
because there is no integer multiple of $m$ in the interval $(ax-1, ax+1)$.
To apply this fact to the evaluation of $\tau_{\alpha\pm\epsilon},$
decompose
Eq.~\eqref{eq:taualphaterms}
to four parts
\begin{alignat}{3}
  \tau_{\alpha\pm\epsilon}(i) = \tau_{\frac{a}{m}\pm\epsilon}(i)
  & = & \ &  m \left(1 - \left\lfloor \frac{ i  (a \pm m \epsilon)  }{ m } \right\rfloor\right) 
  + \sum_{j=1}^{m-1} \left\lfloor \frac{ j ( a \pm m \epsilon)  }{ m } \right\rfloor \\
  &   & \ & + \left\lfloor \frac{ m ( a \pm m \epsilon)  }{ m } \right\rfloor 
  + \sum_{j \in [m] \setminus \{i\}} \left\lfloor \frac{ (i-j) ( a \pm m \epsilon)  }{ m } \right\rfloor,
\end{alignat}
then the second and the fourth parts are invariant with respect to
$\pm \epsilon$ and keep their values at $\epsilon = 0$. The first part is
\begin{equation}
  m \left(1 - \left\lfloor \frac{ i  (a \pm m \epsilon)  }{ m } \right\rfloor\right)
  = \begin{cases}
\displaystyle    m \left(1 - \left\lfloor \frac{ i  a  }{ m } \right\rfloor\right) & (i \neq m) \\
    m \left(1 - \left\lfloor  {a \pm m \epsilon}  \right\rfloor\right) & (i = m)
\end{cases}
  \end{equation}
where
\begin{equation}
  m \left(1 -  \left\lfloor  {a + m \epsilon} \right\rfloor\right)
  = m \left(1 - a \right)
\end{equation}
which keeps the same value with the $\epsilon = 0$ case and
\begin{equation}
  m \left(1 - \left\lfloor {a - m \epsilon} \right\rfloor\right)
  =  {m \left(1 - (a - 1)\right)}
\end{equation}
which is larger than the $\epsilon = 0$ case by $+m$.
The third part is
\begin{equation}
  \left\lfloor \frac{ m ( a \pm m \epsilon)  }{ m } \right\rfloor
  = \left\lfloor    a \pm m \epsilon  \right\rfloor
\end{equation}
which, as seen above,
share the value as the $\epsilon = 0$ case when $\pm$ takes the
positive sign and
takes the value smaller than the $\epsilon=0$ case by $1$
when $\pm$ takes the negative sign.

Combining these,
we have shown that
\begin{equation}
  \tau_{\alpha+\epsilon}(i) = \tau_{\alpha}(i)
  ( = \supermodtwo{m}{ai+1} \text{\ by (i)) \ }
  \text{\ for\ } i \in [m]
\end{equation}
and
\begin{align}
  \tau_{\alpha-\epsilon}(i) & = \tau_{\alpha}(i)
  + m \mychi(i=m) - 1 
   = \supermodtwo{m}{ai+1}-1+m \mychi(i=m) \\
  & = \mathrm{Mod}_m(ai) + m \mychi(i=m) 
   = \supermodtwo{m}{ai}  \text{\ for\ } i \in [m],
\end{align}
which complete the proof.
\end{proof}

\subsection{Proof of Proposition \ref{prop:division}}

\begin{proof}
  Suppose that $\pi$ is non-branching. 
  For $\hat{\theta} \in \myPsi_{m}^{-1} \circ \myPhi_{m}^{-1}(\{\pi\}),$
  take an $\alpha \in \fperm{m}{\hat{\theta}}$ arbitrarily.
  It holds that $\hat{\theta} = \tau_{\alpha}^{m}$. Then, 
  $\pi = \myPhi_{m}\circ\myPsi_{m}(\tau_{\alpha}^{m})$ which
  is $\tau_{\alpha}^{m-1}$ by Proposition \ref{prop:compat}.
  Thus $\alpha \in \fperm{m-1}{\pi}$. We have shown that
  $\fperm{m}{\hat{\theta}} \subset \fperm{m-1}{\pi}.$
  If $\fperm{m-1}{\pi} \setminus \fperm{m}{\hat{\theta}}
  \neq \emptyset$
  then there exists an irrational $\beta$ in this set difference 
  (because it is the difference of two open intervals).
  Then $\tau_\beta^{m} \neq \hat{\theta}$ by $\beta \not\in
  \fperm{m}{\hat{\theta}} $
  and Theorem \ref{known:sur}
  whereas $\tau_\beta^{m-1} = \pi$
  by $\beta \in \fperm{m-1}{\pi}.$
  However,
  $\myPsi^{-1}\circ\myPhi^{-1}(\{\pi\})=\{\hat{\theta}\}$
  by definition,
  contradicting to
  Proposition \ref{prop:compat}. Thus, it holds that
  $\fperm{m}{\hat{\theta}} = \fperm{m-1}{\pi}.$

  Suppose that $\pi$ is branching. Let
  $\CDS_m(\myPhi_m^{-1}(\pi)) = \{a\}$ with $\gcd(a,m)=1$ and
  $\theta_{a,0}, \theta_{a,1} \in \Rcr_m$ be the two children
  of $\myPhi_m^{-1}(\pi)$.

By Proposition \ref{prop:eps},
$\tau_{a/m - \epsilon}^{m} = \theta_{a,0}$ for sufficiently small
$\epsilon > 0$ and
$\tau_{a/m + \epsilon}^{m} = \theta_{a,1}$ for sufficiently small
    $\epsilon \geq 0$, which imply that
    $\fperm{m}{\theta_{a,0}}$
and
    $\fperm{m}{\theta_{a,1}}$
are adjacent
    (in this order) and that $a/m$ is the boundary
of them. To show the inclusion
$\fperm{m}{\theta_{a,0}} \subset \fperm{m-1}{\pi}$,
    take $\alpha \in
    \fperm{m}{\theta_{a,0}}$ arbitrarily. We have $\theta_{a,0} = \tau_{\alpha}^m$ and
    $\pi = \myPhi_m\circ\myPsi_m(\theta_{a,0}) =
    \myPhi_m\circ\myPsi_m(\tau_{\alpha}^m) =
    \tau_{\alpha}^{m-1}$ by Proposition \ref{prop:compat}. Thus
    $\alpha \in \fperm{m-1}{\pi}.$ This shows that $\fperm{m}{\theta_{a,0}} \subset \fperm{m-1}{\pi}$.
Also,
$\{a/m\} \cup
\fperm{m}{\theta_{a,1}} \subset \fperm{m-1}{\pi}$
is shown by taking $\alpha$ in LHS by a similar argument,
noting that $\tau_{a/m}^{m} = \theta_{a,1}$
by Proposition \ref{prop:eps}~(i)
      and that
      $a/m$ is not in the
\Th{$(m-1)$}%
      Farey sequence. We have
      shown that $\fperm{m}{\theta_{a,0}} \cup
      \{a/m\} \cup \fperm{m}{\theta_{a,1}}
      \subset \fperm{m-1}{\pi},$
      where the disjointness of LHS is obvious.
      For the reverse inclusion, suppose that
      $\fperm{m-1}{\pi} \setminus (\fperm{m}{\theta_{a,0}} \sqcup
      \{a/m\} \sqcup  \fperm{m}{\theta_{a,1}}) \neq \emptyset$.
      By taking an irrational $\beta$ in this
      set difference, we have $\tau_\beta^{m-1} = \pi$ and
      $\tau_\beta^{m} \not \in \{\theta_{a,0}, \theta_{a,1}\}$,
      contradicting to that
      $\tau_\beta^{m-1} = \myPhi_m\circ\myPsi_m(\tau_\beta^{m})$
      (by Proposition \ref{prop:compat})
      and
      $\myPsi_m^{-1}\circ\myPhi_m^{-1}({\pi}) =
      \{\theta_{a,0}, \theta_{a,1}\}$.
\end{proof}
\printindex
\end{document}